\documentclass[a4paper, bibliography=totoc]{scrartcl}
\usepackage{amsmath}
\usepackage{amssymb}
\usepackage{enumerate}
\usepackage{amsthm}
\usepackage{todonotes}
\usepackage{mathtools} 
\usepackage[UKenglish]{babel}
\usepackage[T1]{fontenc}
\usepackage[utf8]{inputenc}
\usepackage{lmodern} 
\usepackage{tikz}
\usepackage{thmtools} 
\usepackage{thm-restate} 
\usepackage{stackrel} 
\usepackage[mathlines]{lineno}
\usepackage[normalem]{ulem}
\usepackage{graphicx}

\graphicspath{{./figs}}

\usepackage{scalerel}
\usepackage{tikz}
\usetikzlibrary{svg.path}
\definecolor{orcidlogocol}{HTML}{A6CE39}
\tikzset{
	orcidlogo/.pic={
		\fill[orcidlogocol] svg{M256,128c0,70.7-57.3,128-128,128C57.3,256,0,198.7,0,128C0,57.3,57.3,0,128,0C198.7,0,256,57.3,256,128z};
		\fill[white] svg{M86.3,186.2H70.9V79.1h15.4v48.4V186.2z}
		svg{M108.9,79.1h41.6c39.6,0,57,28.3,57,53.6c0,27.5-21.5,53.6-56.8,53.6h-41.8V79.1z M124.3,172.4h24.5c34.9,0,42.9-26.5,42.9-39.7c0-21.5-13.7-39.7-43.7-39.7h-23.7V172.4z}
		svg{M88.7,56.8c0,5.5-4.5,10.1-10.1,10.1c-5.6,0-10.1-4.6-10.1-10.1c0-5.6,4.5-10.1,10.1-10.1C84.2,46.7,88.7,51.3,88.7,56.8z};
	}
}

\newcommand\orcidicon[1]{\href{https://orcid.org/#1}{\mbox{\scalerel*{
				\begin{tikzpicture}[yscale=-1,transform shape]
					\pic{orcidlogo};
				\end{tikzpicture}
			}{Q}}}}
		
\usepackage[unicode]{hyperref} 
\definecolor{blue3}{rgb}{.1,.0,.4}
\hypersetup{colorlinks=true, linkcolor=red, urlcolor=blue3, citecolor=blue3, pdfpagemode=UseNone, pdfstartview=FitH, bookmarksopen=true, hypertexnames=false} 
\usepackage[open]{bookmark}

\declaretheorem[name=Theorem,numberwithin=section]{thm} 
\newtheorem*{thm*}{Theorem}

\newtheorem*{define*}{Definition}
\newtheorem{define}[thm]{Definition}

\newtheorem*{lemma*}{Lemma}
\newtheorem{lemma}[define]{Lemma}

\newtheorem*{algorithm*}{Algorithm}

\newtheorem*{notation*}{Notation}

\newtheorem*{construction*}{Construction}

\newtheorem*{prop*}{Proposition}

\newtheorem*{obs*}{Observation}

\newtheorem*{fact*}{Fact}

\newtheorem*{remark*}{Remark}

\newtheorem*{quest*}{Question}

\newtheorem*{cor*}{Corollary}

\newtheorem*{conjecture*}{Conjecture}
\newtheorem{conjecture}[define]{Conjecture}

\newtheorem*{question*}{Question}

\newtheorem*{example*}{Example}

\numberwithin{claimcounter}{define}
\newtheorem*{claim*}{Claim}

\newcommand{\R}{\mathbb{R}}

\newcommand{\Z}{\mathbb{Z}}
\newcommand{\N}{\mathbb{N}}

\newcommand{\floor}[1]{\lfloor#1\rfloor}
\newcommand{\ceil}[1]{\lceil#1\rceil}

\newcommand{\rest}[2]{#1|_{#2}} 
\newcommand{\lip}{\operatorname{Lip}}
\newcommand{\bilip}{\operatorname{bilip}}
\DeclareMathOperator{\diam}{diam}
\DeclareMathOperator{\conv}{conv}
\DeclareMathOperator{\id}{id}

\DeclareMathOperator{\proj}{proj}

\newcommand{\abs}[1]{\left|#1\right|}

\newcommand{\enorm}[1]{\left\|#1\right\|}
\newcommand{\norm}[1]{\left\|#1\right\|}

\newcommand{\mc}[1]{\mathcal{#1}}
\newcommand{\mf}[1]{\mathfrak{#1}}

\newcommand{\set}[1]{\left\{#1\right\}}
\newcommand{\br}[1]{\left(#1\right)}
\newcommand{\sqbr}[1]{\left[#1\right]}

\DeclareMathOperator{\dom}{dom}
\DeclareMathOperator{\img}{image}
\newcommand{\upp}[1]{^{(#1)}}

\DeclareMathOperator{\inter}{Int}
\newcommand{\wt}[1]{\widetilde{#1}}

\newcommand{\cl}[1]{\overline{#1}}

\makeatletter
\newcommand{\mylabel}[2]{#2\def\@currentlabel{#2}\label{#1}}

\makeatother 

\title{Extending Bilipschitz Mappings between Separated Nets}
\author{Michael Dymond\thanks{The present work developed from a research visit of M.D. to V.K. at IST Austria, funded by a London Mathematical Society Research in Pairs grant.} \\ \small{School of Mathematics}
	\\ \small{University of Birmingham}
	\\ \small{Birmingham B15 2TT, UK}
	\\ \small{\orcidicon{0000-0002-1900-3549}\ \href{https://orcid.org/0000-0002-1900-3549}{0000-0002-1900-3549}
	}	\and	Vojtěch Kaluža\thanks{This work was done while V.K. was fully funded by the Austria Science Fund (FWF) [M 3100-N].}
	\\ \small{Institute of Science and Technology Austria}
	\\ \small{Am Campus 1}
	\\ \small{3400 Klosterneuburg, Austria}
	\\ \small{\orcidicon{0000-0002-2512-8698}\ \href{https://orcid.org/0000-0002-2512-8698}{0000-0002-2512-8698}
}}
\date{}

\begin{document}
	\maketitle
	\begin{abstract}
		We provide a new characterisation of the decades old open problem of extending bilipschitz mappings given on a Euclidean separated net. In particular, this allows for the complete positive solution of the open problem in dimension two. Along the way, we develop a set of tools for bilipschitz extensions of mappings between subsets of Euclidean spaces.
	\end{abstract}

\section{Introduction}
The purpose of the present article is to lay foundations for an approach to the problem of whether every bilipschitz mapping of a separated net of $\R^{d}$ to $\R^{d}$ admits a bilipschitz extension to the whole of $\R^{d}$: a 22 year old open question (in all dimensions $d\geq 2$) posed by Alestalo, Trotsenko and  Väisälä in \cite{bilip_ext_sep_nets}. Indeed the framework provided in the present work allows for the complete positive solution of this bilipschitz extension problem in dimension 2. In a companion article~\cite{DK_2dim} we perform in detail various two-dimensional constructions which, together with the hereafter established results, provide for any given $L$-bilipschitz mapping of a separated net of $\R^2$ to $\R^2$, a $C(L)$-bilipschitz extension to the whole of $\R^2$. This moreover delivers a complete positive solution of one of the listed problems~\cite[Problem~2.6.1]{adiceam2016open}, posed by Navas, at the Arbeitsgemeinschaft on Mathematical Quasicrystals in Oberwolfach in 2015. The exposure and relevance of the bilipschitz extension problem for separated nets is further evidenced by its explicit references in \cite[Remark~7]{cortez2016some} and \cite[p.6205]{smilansky_solomon_2022}. 

There is a pantheon of mathematical research going back to the early 20th century on the general problem of extending a given mapping of a subset of an ambient space whilst preserving its key properties. We highlight the famous extension theorems of Tietze--Urysohn~\cite{Tietze--Urysohn} for continuous and Whitney~\cite{whitney} for smooth functions. The class of Lipschitz functions/mappings lies between the continuous and smooth classes. Due primarily to Rademacher's theorem~\cite{rademacher1919partielle}, the Lipschitz condition suffices in Euclidean settings to prevent the pathological behaviour permissible under just continuity. On the other hand the Lipschitz condition allows for much greater flexibility in comparison with smooth: note that important natural mappings such as the Euclidean norm (and many other norms and related functions such as distance functions and projections) are Lipschitz but not smooth. Unlike both continuity or differentiability, the Lipschitz condition transfers seamlessly to discrete settings as well as to general metric spaces. Further, Lipschitz mappings play an important role in the theory of discrete metric spaces, for example in the study of embeddings~\cite{Matousek_MetricEmb, Ostrovskii_MetricEmb} or in the notion of bilipschitz equivalence, e.g., \cite{BK1, McM}.

Lipschitz mappings are well extendable in some key settings. Kirszbraun's theorem~\cite{Kirszbraun,Valentine}, also known as the Kirszbraun--Valentine Theorem, states that any Lipschitz mapping of a subset of Hilbert space to another Hilbert space may be extended to the whole space, even without increasing its Lipschitz constant. Moreover, McShane's Extension Theorem~\cite{mcshane_extension} similarly provides Lipschitz extensions with the same Lipschitz constant for any (real-valued) Lipschitz function of a subset of a metric space. Kirszbraun's theorem continues to shape the modern field; we point to several significant improvements, such as \cite{EvaKopecka2012}, \cite{Azagra2021} and \cite{Ciosmak2024}.

That the analogue of Kirszbraun's theorem for bilipschitz mappings fails is, after some thought, obvious, and the reason is topological. For example we may consider a bilipschitz mapping of the unit sphere and the origin, which maps the origin `outside' the image of the sphere. Such a mapping cannot even be extended to a homeomorphism of the space, let alone to a bilipschitz one. Thus research on the bilipschitz extension problem has focussed on the question of when we may extend a given bilipschitz mapping of a subset. From what kind of subsets is this possible? Most existing positive results occur in dimension two; see \cite{Tukia_planar_Schoenflies, tukia1981extension, DP_bilip_square, Kovalev_bilip_ext_from_line, Kovalev_optimal_bilip} which, in particular, provide bilipschitz extensions from a line and from the boundary of a square (or circle). We will discuss the two dimensional situation in more detail in the companion article~\cite{DK_2dim}.

The challenge of bilipschitz extension in higher dimensional spaces is made evident by the number of pathological examples and notorious enduring open problems connected to extendability of mappings between higher dimensional spaces. The Jordan--Schoenflies Theorem (see, e.g., \cite[Thm.~3.1]{Thomassen-JordanSchoenflies}) states that any embedding of $S^1$ into $\R^2$ may be extended to a homeomorphism of $\R^{2}$. However, this statement fails for embeddings of $S^{d-1}$ into $\R^{d}$ in all higher dimensions $d\geq 3$: in dimension 3 a counterexample is provided by Alexander's Horned Sphere~\cite{alexander_horned_sphere}. Shockingly the statement still fails if the given embedding is also assumed to be bilipschitz. A counterexample is given by the Fox--Artin Wild Sphere~\cite[Theorem~3.7]{Martin_wild_ball}. 

Using the fact that the sphere $S^{d-1}$ minus a point is homeomorphic to $\R^{d-1}$, the Jordan--Schoenflies theorem can be reformulated for embeddings of $\R$ into $\R^{2}$ instead of $S^{1}$ into $\R^{2}$. Further, when the given embedding of $\R$ is $L$-bilipschitz, the extension to $\R^{2}$ can be made $C(L)$-bilipschitz \cite{Tukia_planar_Schoenflies, Kovalev_bilip_ext_from_line}. This evolution of the Jordan--Schoenflies Theorem has particular relevance to the problem of extending a given bilipschitz mapping $f\colon X\to\R^{d}$ of a separated net $X$ of $\R^{d}$.

Theorem~\ref{thm:Zd_equiv} of the present paper reduces the extension problem for a general separated net $A$ to that of $A=\Z^d$. After that, our strategy proposes that the next step in the construction of an extension of $f$ should be to extend $f$ initially to a $(d-1)$-dimensional hyperplane $\mc{H}$. More precisely we seek to construct a bilipschitz extension $F\colon \Z^d\cup \mc{H}\to\R^{d}$ of $f$.
One significant difficulty to overcome in this task is to ensure that the image $F(\mc{H})$ separates the space in the correct way: points $x,y$ of $\Z^d$ must belong to the same connected component of $\R^{d}\setminus \mc{H}$ if and only if their images $f(x)$ and $f(y)$ belong to the same connected component of $\R^{d}\setminus F(\mc{H})$.
However, in dimensions higher than two there is yet a further challenge: Even if we succeed in constructing the bilipschitz extension $F$ achieving a correct separation of image points, there is no guarantee that even just $F|_{\mc{H}}$ permits a bilipschitz extension to $\R^{d}$, let alone one which also extends $f$.
In two dimensions, we know from the bilipschitz version of the Jordan--Schoenflies Theorem for embeddings of lines~\cite{Tukia_planar_Schoenflies, Kovalev_bilip_ext_from_line} that $F|_{\mc{H}}$ can be extended to a bilipschitz mapping of $\R^{2}$: a fact we exploit critically, along with the results of the present paper, to obtain the positive solution of the extension problem in dimension $2$ in \cite{DK_2dim}.  

To overcome the obstacles remaining in the higher dimensional setting will require significant future research. The present article aims to be a start along this path. We now state our main results and explain how they can form part of a plausible programme to attack the higher dimensional problem.

\subsection*{Main Results}
Alestalo, Trotsenko and V\"ais\"al\"a~\cite[4.4(ii)]{bilip_ext_sep_nets} pose the questions, in dimensions $d\geq 2$, of the bilipschitz extendability of a given bilipschitz mapping $f$ of 
\begin{enumerate}[(a)]
	\item\label{ATV1} $\Z^{d}$ to $\R^{d}$,
	\item\label{ATV2} a given separated net of $\R^{d}$ to $\R^{d}$.
\end{enumerate}
Whilst \eqref{ATV1} is clearly a subquestion of \eqref{ATV2}, \cite{bilip_ext_sep_nets} points out that any further relationship between these questions is unknown and far from obvious. The naive argument to try to show that questions \eqref{ATV1} and \eqref{ATV2} are equivalent would require, for a given separated net $A\subseteq \R^{d}$, there to be a bilipschitz bijection between $A$ and $\Z^{d}$. However, there exist separated nets not admitting any such bijection; these are called `non-rectifiable' or `not bilipschitz equivalent to the lattice' in the literature; see~\cite{BK1,McM}. Despite this occurrence of `very different' separated nets of $\R^{d}$, our first main result reduces, in all dimensions, the bilipschitz extension problem for all separated nets to that of the integer lattice, establishing the equivalence of questions \eqref{ATV1} and \eqref{ATV2} of  Alestalo, Trotsenko and V\"ais\"al\"a~\cite[4.4(ii)]{bilip_ext_sep_nets}. 
\begin{thm}\label{thm:Zd_equiv}
	Let $d\in \N$. Then the following are equivalent:
	\begin{enumerate}[(i)]
		\item\label{Zd} Every bilipschitz mapping $f\colon \Z^{d}\to \R^{d}$ admits a bilipschitz mapping $F\colon \R^{d}\to\R^{d}$ with $F|_{\Z^{d}}=f$.
		\item\label{gen_sep_net} For every separated net $A$ of $\R^{d}$ and every bilipschitz mapping $f\colon A\to \R^{d}$ there is a bilipschitz mapping $F\colon \R^{d}\to\R^{d}$ with $F|_{A}=f$.
	\end{enumerate}
	Furthermore, if \eqref{Zd} holds with $\bilip(F)\leq C_{d}\br{\bilip(f)}$ for some monotone increasing function $C_{d}\colon [1,\infty)\to [1,\infty)$ then \eqref{gen_sep_net} holds with
	\begin{equation*}
	\bilip(F)\leq K\cdot C_{d}\br{24\sqrt{d}R^{2}K^{3}\bilip(f)},
	\end{equation*}
	where $K:= 16\max\set{\frac{3d}{r},1}$ and $R, r$ denote the net and separation constants of $A$ respectively. 
\end{thm}

Our second main result characterises the bilipschitz extension problem on the integer lattice in terms of existence of an `intermediate extension' to a sequence of `horizontal' hyperplanes.
\begin{thm}\label{thm:bil_ext_equiv}
	Let $d\in\N$, $d\geq2$ and $f\colon\Z^{d}\to\R^{d}$ be a bilipschitz mapping.
	Then the following are equivalent:
	\begin{enumerate}[(i)]
		\item\label{bil_ext} There exists a bilipschitz mapping $F\colon \R^{d}\to \R^{d}$ such that $F|_{\Z^{d}}=f$.
		\item\label{bil_wsep} There exist $T\in\N$ and a bilipschitz mapping $G\colon \R^{d}\to\R^{d}$ such that the mapping
		\begin{equation*}
		\wt{F}\colon \Z^{d}\cup \br{\R^{d-1}\times\br{T\Z+\frac{1}{2}}}\to \R^{d}, \qquad \wt{F}(x)=\begin{cases}
		f(x) & \text{if }x\in\Z^{d},\\
		G(x) & \text{if }x\in \R^{d-1}\times\br{T\Z+\frac{1}{2}},
		\end{cases}
		\end{equation*}
		is bilipschitz and for every $k\in\Z$ it holds that
		\[
		f\br{\Z^d\cap \br{\R^{d-1}\times\left[(k-1)T+\frac{1}{2},kT+\frac{1}{2}\right]}}\subseteq G\br{\R^{d-1}\times\left[(k-1)T+\frac{1}{2},kT+\frac{1}{2}\right]}.
		\]
	\end{enumerate}
	Furthermore, whenever \eqref{bil_wsep} holds with $M_{1}\geq \bilip(\wt{F})$ and $M_{2}\geq \bilip(G)$, the bilipschitz extension $F$ of $f$ in \eqref{bil_ext} may be found with
	\begin{equation*}
	\bilip(F)\leq d^{125d}M_{1}^{26d}M_{2}^{27d}T^{15d}.
	\end{equation*}
\end{thm}

Finally, our last main result demonstrates the potential of Theorem~\ref{thm:bil_ext_equiv} to form part of a complete positive solution of the bilipschitz extension problem for separated nets in higher dimensions. More precisely, we use Theorem~\ref{thm:bil_ext_equiv} to relate this extension problem to two conjectures about extendability to or from hyperplanes. In particular, we show that proving these two conjectures suffices to give the full positive solution. The first conjecture `extending from $\Z^{d}$ to a hyperplane' is clearly a weaker statement than the target of `extending from $\Z^{d}$ to the whole of $\R^{d}$'.
\begin{conjecture}\label{conj1}
	Let $d\in \N$, $d\geq 2$, $L\geq 1$ and $f\colon \Z^{d}\to \R^{d}$ be an $L$-bilipschitz mapping. Then there exist constants $K:=K(d,L)>L$, $J:=J(d,L)>L$, determined by $d$ and $L$, and a $K$-bilipschitz mapping $G\colon \R^{d}\to\R^{d}$ such that the mapping 
	\begin{equation*}
	F\colon \Z^{d}\cup\br{\R^{d-1}\times \set{\frac{1}{2}}}\to \R^{d},\qquad F(x)=\begin{cases}
	f(x) & \text{if }x\in\Z^{d},\\
	G(x) & \text{if }x\in\R^{d-1}\times\set{\frac{1}{2}}
	\end{cases}
	\end{equation*}
	is $J$-bilipschitz and for any $x,y\in\Z^{d}$ we have
	\begin{align*}
	\text{$x$ and $y$ belong to the same} &\text{ connected component of $\R^{d}\setminus \br{\R^{d-1}\times\set{\frac{1}{2}}}$}\\
	&\Updownarrow\\
	\text{$f(x)$ and $f(y)$ belong to the same} &\text{ connected component of $\R^{d}\setminus F\br{\R^{d-1}\times\set{\frac{1}{2}}}$}.
	\end{align*}
\end{conjecture} 
The second conjecture is about extending very special kinds of bilipschitz mappings defined on two parallel hyperplanes:
\begin{conjecture}\label{conj2}
	Let $d\in\N$, $d\geq 2$, $L,M\geq 1$ and $G\colon \R^{d}\to\R^{d}$ be an $L$-bilipschitz mapping such that the mapping $\psi\colon\R^{d-1}\times\set{0,1}\to\R^{d}$ defined by
	\begin{equation*}
	\psi(x)=\begin{cases}
	x & \text{if }x\in\R\times\set{0},\\
	G(x) & \text{if }x\in\R\times\set{1},
	\end{cases}
	\end{equation*}
	is $M$-bilipschitz. Then there exist a constant $C:=C(d,L,M)$, determined by $d$, $M$ and $L$, and a $C$-bilipschitz mapping $\Psi\colon\R^{d}\to\R^{d}$ extending $\psi$.
\end{conjecture}
We can now state our last main result, which establishes Conjectures~\ref{conj1} and \ref{conj2} as a route to bilipschitz extension from separated nets in any dimension:
\begin{thm}\label{thm:conjectures_imply_full_solution}
	Suppose that $d\in\N$, $d\geq 2$ and that Conjecture~\ref{conj2} holds for $d$. Then the following are equivalent:
	\begin{enumerate}[(i)]
		\item\label{f_has_extension_F} For any $L\geq 1$ and any $L$-bilipschitz mapping $f\colon \Z^{d}\to \R^{d}$ there is a constant $P:=P(d,L)$, determined by $d$ and $L$, and a $P$-bilipschitz mapping $F\colon \R^{d}\to\R^{d}$ such that $F|_{\Z^{d}}=f$.
		\item\label{conj1_holds} Conjecture~\ref{conj1} holds for $d$.
	\end{enumerate}
	Furthermore, whenever \eqref{conj1_holds} holds, then the constant $P(d,L)$ in \eqref{f_has_extension_F} can be chosen such that
	\begin{equation*}
	P(d,L)\leq d^{125d}J^{26d}\br{KC}^{27d}\br{\frac{K\sqrt{d}(1+JL)}{K-L}+1}^{15d},
	\end{equation*}
	where $J:=J(d,L)$, $K:=K(d,L)$, $C:=C(d,K^{2},K^{2})$ are defined according to Conjectures~\ref{conj1} and \ref{conj2}.
\end{thm}
We remark that the \eqref{f_has_extension_F}$\Rightarrow$\eqref{conj1_holds} implication in Theorem~\ref{thm:conjectures_imply_full_solution} is trivial and independent of Conjecture~\ref{conj2}.
\subsection*{Remarks on Conjectures~\ref{conj1} and \ref{conj2}}
We presently outline a `two-step' strategy for proving Conjecture~\ref{conj1}; we emphasise that while we complete both of the proposed steps in the case $d=2$ in \cite{DK_2dim}, for $d\geq 3$ both steps remain currently open.

For convenience, we denote the hyperplane $\R^{d-1}\times\set{\frac{1}{2}}$, by $\mc{H}$. Assume we are given a bilipschitz mapping $f\colon\Z^d\to\R^d$. We suggest the following steps:
\begin{enumerate}[(1)]
	\item\label{hyperplane_ext} Construct a bilipschitz extension $F\colon \Z^d\cup\mc{H}\to\R^d$ of $f$ in such a way that there is a bilipschitz mapping $G\colon\R^d\to\R^d$ so that $\rest{G}{\mc{H}}=\rest{F}{\mc{H}}$.
	
	\item\label{correct_hyperplane} Adjust the extension $F$ so that, in addition to its property from \eqref{hyperplane_ext}, it separates the points of $\Z^d$ correctly, i.e., $x,y\in \Z^d$ belong to the same connected component of $\R^d\setminus \mc{H}$ if and only if $F(x), F(y)$ belong to the same connected component of $\R^d\setminus F\br{\mc{H}}$.
\end{enumerate}

In Step~\eqref{correct_hyperplane} we can make use of the easy observation that, following Step~\eqref{hyperplane_ext}, there will be a parameter $P=P(\bilip(F))>0$ such that whenever $x,y\in \Z^d$ belong to the same connected component of $\R^d\setminus \mc{H}$ and $F(x), F(y)$ belong to different connected components of $\R^d\setminus F\br{\mc{H}}$ (or vice versa), then both $x, y$ are at distance at most $P$ from $\mc{H}$. Thus, transforming the picture by applying $G^{-1}$ from Step~\eqref{hyperplane_ext}, we get that the image points on the `wrong side' of $G^{-1}\circ F(\mc{H})=\mc{H}$ all occur within a horizontal slab around $\mc{H}$ of thickness determined by $\bilip(G^{-1}\circ F)$. The modification in Step~\eqref{correct_hyperplane} can therefore take place within this horizontal slab; see \cite[Theorem~7.1]{DK_2dim} for the $2$-dimensional analogue of this.
We believe that our 2-dimensional construction should also work in general dimension; however, there are significant technical difficulties in describing the relevant local reparametrizations of multi-dimensional mappings formally.

Our strategy from \cite[Sec.~4]{DK_2dim} to complete Step~\eqref{hyperplane_ext} for $d=2$ seems to be very specific to dimension $2$. Unfortunately, we currently do not know how to extend it to higher dimensions. We only remark that, again in an analogy with the spherical setting, a sufficient condition ensuring the existence of the mapping $G$ from Step~\eqref{hyperplane_ext} should be the local bilipschitz flatness of $\rest{F}{\mc{H}}$; we refer to the resolution of the so-called Schoenflies problem in the bilipschitz category in all dimensions $d$ from \cite[Thm.~7.8]{Elems_of_Lip_topology}.

In relation to Conjecture~\ref{conj2}, the theory of bilipschitz extensions in general dimension is currently developed for mappings defined on spheres \cite{Sullivan, Tukia_Vaisala-approx_and_extend, Elems_of_Lip_topology}, but not for those defined on hyperplanes. Although the two settings are often very closely related through the stereographic projection, the latter is only locally bilipschitz, and thus, inconsequential in a bilipschitz setting. However, there are promising indications from the positive results available for the spherical case and one can follow the (highly sophisticated) proofs and try to obtain the analogous results for hyperplanes.

Building upon the work of Sullivan~\cite{Sullivan}, Tukia and V\"ais\"al\"a~\cite[Thm.~3.12]{Tukia_Vaisala-approx_and_extend} proved a bilipschitz version of the so-called annulus conjecture in all dimensions $d$. Their theorem states that the closed region bounded by two bilipschitz and \emph{locally bilipschitz flat}\footnote{A mapping $g\colon \partial B(0, 1)\to\R^d$ is locally bilipschitz flat if for every $y\in g(\partial B(0,1))$ there is a neighbourhood $U$ and a bilipschitz homeomorphism of pairs $\br{U, U\cap g(\partial B(0,1))}\to \br{B(0, 1), B(0,1)\cap \mc{H}_0}$.} embeddings of $(d-1)$-spheres in $\R^{d}$, with one of them lying inside the other, is bilipschitz homeomorphic to the standard annulus $\cl{B}(0,1)\setminus B(0,1/2)$. Moreover, as a corollary, they proved \cite[Thm.~3.14]{Tukia--Vaisala-Quasiconformal_implies_bilip} that for any sense-preserving bilipschitz mapping $h\colon\partial B(0,1)\to\partial B(0,1)$ and any $r\in(0,1)$ there is a bilipschitz mapping $H\colon \cl{B}(0,1)\to \cl{B}(0,1)$ extending $h$ such that $\rest{H}{B(0, r)}$ is the identity.

Speculatively assuming that the analogue of the bilipschitz annulus theorem holds true also when spheres are replaced by hyperplanes and annuli by slabs between two hyperplanes, we can use it in the setting of Conjecture~\ref{conj2} to obtain a bilipschitz mapping $\Phi\colon\R^d\to\R^d$ so that $\Phi\circ \psi(\R^{d-1}\times\set{i})=\R^{d-1}\times\set{i}$ for each $i\in\set{0,1}$; the  assumption that both `parts', $\psi|_{\R^{d-1}\times\set{0}}$ and $\psi|_{\R^{d-1}\times\set{1}}$, of $\psi$ extend to bilipschitz mappings of $\R^{d}$ ensures that the local bilipschitz flatness condition from the annulus theorem is satisfied.
Assuming further that also an analogue of the corollary \cite[Thm.~3.14]{Tukia--Vaisala-Quasiconformal_implies_bilip} mentioned above holds true with the sphere replaced by a hyperplane and the ball by a half-space, the bilipschitz mapping $\Phi\circ \psi|_{\R^{d-1}\times\set{0,1}}$ can be extended to a bilipschitz mapping $\wt{G}\colon\R^d\to\R^d$. Then $\Phi^{-1}\circ\wt{G}$ gives the mapping $G$ sought after in Conjecture~\ref{conj2}.

\subsection*{Strategy for the proofs of the main theorems}
Returning to the main results of the present paper, we describe the big picture of our proofs of Theorems~\ref{thm:Zd_equiv} and \ref{thm:bil_ext_equiv}, picking out the key tools and stepping stones established in the sections that follow. We don't discuss the proof of Theorem~\ref{thm:conjectures_imply_full_solution} here; this is given in Section~\ref{sec:conjectures}.

Ultimately, every bilipschitz extension constructed in the present paper will be obtained by gluing and composing very simple bilipschitz transformations. These `building blocks' will be bilipschitz mappings designed to perform a swap of two specified points, denoted $x$ and $y$ below, whilst disturbing the rest of the space as little as possible. The next lemma, proved in Section~\ref{sec:proto}, provides these mappings with an estimate on their bilipschitz constant: 
\begin{restatable*}{lemma}{restatespin}\label{lemma:tube_spin}
	Let $d\in\N$, $d\geq 2$, $x,y\in\R^{d}$ and $0< r\leq \frac{\enorm{y-x}}{2}$. Then there exists a $\frac{4\enorm{y-x}^{2}}{r^{2}}$-bilipschitz mapping $\tau\colon \R^{d}\to \R^{d}$ such that
	\begin{enumerate}[(i)]
		\item\label{tau1} $\tau(z)=z$ for all $z\in \R^{d}\setminus B([x,y],r)$.
		\item\label{tau2} $\tau(y)=x$ and $\tau(x)=y$.
	\end{enumerate}
\end{restatable*}
Significantly, Lemma~\ref{lemma:tube_spin} allows for the simultaneous bilipschitz swapping of a family of specified partners $(x_{i},y_{i})$, provided that the line segments $[x_{i},y_{i}]$ have some uniform separation. Since the mappings $\tau$ provided by the lemma coincide with the identity outside of a small neighbourhood of $[x,y]$, we may easily glue many of these mappings together in order to perform a `separated' family of swaps simultaneously, incurring only the bilipschitz constant associated with the `most expensive' of the single swaps. This is made precise by the next lemma, also from Section~\ref{sec:proto}:
\begin{restatable*}{lemma}{restatesimpleglue}\label{lemma:simple_gluing}
	Let $I$ be a set and $(A_{i})_{i\in I}$ be a collection of pairwise disjoint open sets. Let $C\geq 1$ and $f\colon \R^{d}\to \R^{d}$ be a mapping such that the restriction $f|_{\R^{d}\setminus \bigcup_{i\in I}A_{i}}$ is $C$-bilipschitz with $f\br{\R^{d}\setminus \bigcup_{i\in I}A_{i}}=\R^{d}\setminus \bigcup_{i\in I}A_{i}$, and for every $i\in I$ the restriction $f|_{\cl{A_{i}}}$ is $C$-bilipschitz with $f(\cl{A_i})=\cl{A_i}$. Then $f$ is $C$-bilipschitz. 
\end{restatable*}
We can now outline a sketch proof of Theorem~\ref{thm:Zd_equiv}:
\begin{proof}[Proof of Theorem~\ref{thm:Zd_equiv} (Sketch)]
	Theorem~\ref{thm:Zd_equiv} is consequence of two facts:
	\begin{enumerate}[(a)]
		\item\label{i} Any $r$-separated set $A\subseteq \R^{d}$ may be transformed via a $P(d,r)$-bilipschitz mapping $\Phi\colon\R^{d}\to\R^{d}$ so that $\Phi(A)\subseteq \Z^{d}$.
		\item\label{ii} For any $R$-net $Y\subseteq \Z^{d}$ and any $L$-bilipschitz mapping $f\colon Y\to \R^{d}$ there is a $P'(d,R,L)$-bilipschitz extension of $f$ to $\Z^{d}$.
	\end{enumerate}
	To verify the only non-trivial implication, \eqref{Zd}$\Rightarrow$\eqref{gen_sep_net}, of Theorem~\ref{thm:Zd_equiv}, let $f\colon A\to \R^{d}$ be a bilipschitz mapping, $\Phi$ be given by \eqref{i} and then let \eqref{ii} provide the bilipschitz extension $g\colon \Z^{d}\to \R^{d}$ of $f\circ \Phi^{-1}\colon Y\to \R^{2}$ with $Y=\Phi(A)\subseteq \Z^{2}$. The desired extension of $f\colon A\to \R^{d}$ is then given by $G\circ \Phi$ where $G$ is the bilipschitz extension of $g$ given by \eqref{Zd}.
	
	The proof of \eqref{i} is based on Lemmas~\ref{lemma:tube_spin} and \ref{lemma:simple_gluing}. First we blow up the separated set $A$, so that its separation becomes large in comparison to $\sqrt{d}$. This ensures that every point of the transformed set $A'$ has a unique choice of nearest point in $\Z^{d}$ and that these pairs of neighbours are well separated from each other. We then perform simultaneous swappings of the points of $A'$ and their chosen nearest neighbours in $\Z^{d}$, by gluing together mappings $\tau$ from Lemma~\ref{lemma:tube_spin} using Lemma~\ref{lemma:simple_gluing}.
	
	For \eqref{ii} we calculate a radius $s$ so that the balls of radius $s$ around each point of $f(Y)$ are pairwise disjoint with some separation. Next we observe that the collection of balls $(B(y,R))_{y\in Y}$ cover $\R^{d}$ and we can estimate the total number of integer lattice points in each such ball. It remains to calculate a large enough integer $N$ so that for any $y,y'\in Y$ the set $\br{f(y)+\frac{1}{N}\Z^{d}}\cap B(f(y),s)$ contains more points than in $\Z^{d}\cap B(y',R)$. This allows for an injective mapping from $\Z^{d}\setminus Y$ to $\R^{d}\setminus f(Y)$ with the additional property that any $x$ in the domain admits $y\in Y$ with $x\in B(y,R)$ and $f(x)\in \br{f(y)+\frac{1}{N}\Z^{d}}\cap B(f(y),s)$. An arbitrary choice of such mapping determines a suitable extension of $f$.
\end{proof}

The proof of Theorem~\ref{thm:bil_ext_equiv} will be based on the following theorem, proved in Section~\ref{s:threading}:
\begin{restatable*}{thm}{restatethread}\label{thm:thread}
	Let $d,H\in\N$ with $d\geq 2$, $L\geq 1$ and 
	\begin{equation*}
	f\colon \br{\R^{d-1}\times\set{\frac{1}{2},H+\frac{1}{2}}}\cup\br{\Z^{d-1}\times\set{1,\ldots,H}}\to\R^{d}
	\end{equation*}
	be an $L$-bilipschitz mapping satisfying
	\begin{equation*}
	f(x)=x\qquad \text{for all }x\in \R^{d-1}\times \set{\frac{1}{2},H+\frac{1}{2}}
	\end{equation*} 
	and
	\begin{equation*}
	f\br{\Z^{d-1}\times\set{1,\ldots,H}}\subseteq \R^{d-1}\times\sqbr{\frac{1}{2},H+\frac{1}{2}},
	\end{equation*}
	Then there is a bilipschitz extension $F\colon\R^{d}\to \R^{d}$ of $f$ with
	\begin{equation*}
	\bilip(F)\leq d^{125d}L^{26d}H^{15d}.
	\end{equation*}
\end{restatable*}
\begin{proof}(Sketch)
	We aim to construct a bilipschitz mapping $\Theta\colon \R^{d}\to\R^{d}$ such that
	\begin{enumerate}[(i)]
		\item $\Theta(x)=x$ for all $x\in\R^{d-1}\times \br{\R\setminus \br{\frac{1}{2},H+\frac{1}{2}}}$, and
		\item $\Theta\circ f(x,m)=(x,m)$ for all $m\in [H]$ and $x\in\Z^{d-1}$.
	\end{enumerate}
	The bilipschitz $F\colon\R^{d}\to\R^{d}$ verifying Theorem~\ref{thm:thread} may then be defined by $F:=\Theta^{-1}$. The required transformation $\Theta$ will be built by composing and gluing together mappings of the form of $\tau$ from \ref{lemma:tube_spin}. By composing and gluing together these `swapping transformations' $\tau$ we first construct a bilipschitz mapping $\Psi$ which approximately projects each `image layer' $f(\Z^{d-1}\times\set{m})$ into a finer rescaling of the lattice $\Z^{d-1}\times\set{m}$, more specifically $\frac{1}{N}\Z^{d-1}\times\set{m}$ for some quantity $N\in\N$ controlled by $d$, $L$ and $H$. This is achieved by Lemma~\ref{lemma:inj_rounding}. 
	
	Construction of $\Psi$: First we perform simultaneously a family of `horizontal' swaps between points $f(x,m)$ for $x\in \Z^{d-1}$, $m\in[H]$ and nominated partners, say `$g(x,m)$', where $g(x,m)$ is a chosen nearby point with the same last coordinate as $f(x,m)$ and such that the first $(d-1)$-coordinates of $g(x,m)$ define a point in $\frac{1}{N}\Z^{d-1}$. Crucially this is done in such a way that the mapping $f(x,m)\mapsto g(x,m)$ is injective, so the points $g(x,m)$ are separated at least by distance $1/N$. Then we perform simultaneously a family of `vertical' swaps of the points $g(x,m)$ and chosen points `$h(x,m)$', where $h(x,m)$ has the same first $(d-1)$-coordinates as $g(x)$, but has last coordinate $m$.

	After applying $\Psi$ our task is to then shuffle the points of each lattice $\frac{1}{N}\Z^{d-1}\times\set{m}$ so that each point $\Psi\circ f(x,m)$ ends up at $(x,m)$.
	This is again accomplished by a bilipschitz mapping $\Upsilon\colon\R^{d}\to\R^{d}$ based on the `swapping transformations' of Lemma~\ref{lemma:tube_spin}; the spacing between each of the hyperplanes $\R^{d-1}\times\set{m}$ allows us to treat each lattice $\frac{1}{N}\Z^{d-1}\times\set{m}$ independently in the construction of $\Upsilon$. Finally, we will set $\Theta=\Upsilon \circ \Psi$. 
	
	Construction of $\Upsilon$: Since $\Psi$ is approximately a projection of each $f(\Z^{d-1}\times\set{m})$ into $\frac{1}{N}\Z^{d-1}\times\set{m}$ and the bilipschitz mapping $f$ is the identity on $\R^{d-1}\times\set{\frac{1}{2},H+\frac{1}{2}}$, the distance between $\Psi\circ f(x,m)$ and $(x,m)$ for $x\in\Z^{d-1}$, $m\in [H]$ is uniformly bounded above by a quantity $T$ determined by $L$, $d$ and $H$. In a sequence of three swapping transformations from Lemma~\ref{lemma:tube_spin}, we move each point $f(x,m)$ to $(x,m)$: The first swap moves $\Psi\circ f(x,m)\in \frac{1}{N}\Z^{d-1}\times\set{m}$ vertically up to the hyperplane of height $m+\delta(x)$, where $\delta(x)>0$ is chosen carefully. The second transformation shifts the new point at most distance $T$ along the hyperplane, to the point $(x,m+\delta(x))$. Finally, the last swap moves the new point back down to level $m$ and to the point $(x,m)$. In order to control the bilipschitz constant in this procedure, it is crucial that for all $x\in\Z^{d-1}$, the first swaps are performed simultaneously, the second swaps are performed simultaneously and then the third swaps are performed simultaneosly. For the first and third swaps, this is easy to achieve, since the vertical projection of the set of points involved is uniformly separated. To achieve a similar `separation' for the second swaps, we need to choose the quantity $\delta(x)>0$ carefully, so that whenever $x,y\in\Z^{d-1}$ are relatively close so that the second swap for $x$ and $y$ `might interact', we have that $\abs{\delta(x)-\delta(y)}$ is uniformly bounded below. This puts the second swaps for $x$ and $y$ onto distinct vertical levels with some minimum separation. A suitable choice of $\delta(x)$ for $x\in\Z^{d-1}$ corresponds to a proper colouring, provided by Brooks' Theorem~\cite{Brooks_1941}, of a certain infinite graph, with vertex set $\Z^{d-1}$.
\end{proof}

With Theorem~\ref{thm:thread} and Lemma~\ref{lemma:simple_gluing} we can already give the full proof of Theorem~\ref{thm:bil_ext_equiv}:

\begin{proof}[Proof of Theorem~\ref{thm:bil_ext_equiv}]
	It is clear that \eqref{bil_ext} implies \eqref{bil_wsep}.
	
	Let us now suppose that $T$, $G$ and $\wt{F}$ are given by \eqref{bil_wsep} and suppose $M_{1}\geq \bilip(\wt{F})$ $M_{2}\geq \bilip(G)$. For each $k\in\Z$ we set $A_k:=\R^{d-1}\times\br{(k-1)T+\frac{1}{2},kT+\frac{1}{2}}$ and apply Theorem~\ref{thm:thread} (appropriately shifted) to the $M_{1}M_{2}$-bilipschitz mapping $G^{-1}\circ \wt{F}$ restricted to the set
	\begin{equation*}
	V_{k}:=\partial A_k\cup\br{\Z^{d-1}\times \set{(k-1)T+1,(k-1)T+2,\ldots,kT}}.
	\end{equation*}
	We obtain for each $k\in\Z$ a bilipschitz mapping $F_{k}\colon\R^d\to \R^{d}$ such that 
	\begin{equation*}
	F_{k}(x)=G^{-1}\circ \wt{F}(x) \quad \text{for all }x\in V_{k},
	\quad\text{and}\quad
	\bilip(F_{k})\leq d^{125d}\br{M_{1}M_{2}}^{26d}T^{15d}=:C.
	\end{equation*}
	We define a mapping $\Phi\colon\R^d\to\R^d$ by prescribing
	\[
	\Phi(x):=F_k(x)\qquad\text{ whenever } x\in \R^{d-1}\times\sqbr{(k-1)T+\frac{1}{2},kT+\frac{1}{2}},\,k\in\Z.
	\]
	The mapping $\Phi$ is well-defined, since $\rest{F_k}{\R^{d-1}\times\set{(k-1)T+\frac{1}{2},kT+\frac{1}{2}}}$ is equal to the identity for every $k\in\Z$.
	
	Applying Lemma~\ref{lemma:simple_gluing} to $\Phi$, the collection of open sets $\br{A_k}_{k\in\Z}$ and $C$ defined above, we establish that $\Phi$ is $C$-bilipschitz. To complete the proof, it only remains to set $F:=G\circ \Phi$ and note that $\bilip(F)\leq CM_{2}\leq d^{125d}M_{1}^{26d}M_{2}^{27d}T^{15d}$.
\end{proof}

\section{Notation}
Much of our notation is shared by the companion article~\cite{DK_2dim}. The notation $[n]$ refers to $\set{1,\ldots, n}$ for $n\in\N$. The standard euclidean norm on $\R^d$ is denoted by $\norm{\cdot}$. 
For $x,y\in\R^d$, the closed line segment with endpoints $x$ and $y$ is denoted by $[x,y]$.

We write $e_1, \ldots, e_d$ for the standard orthonormal basis vectors of $\R^d$. The orthogonal projection to the $i$-th coordinate in $\R^d$ is denoted by $\proj_i$ and by $\proj_{\R^{d-1}}$ we mean the orthogonal projection to the first $d-1$ coordinates.

\paragraph{Sets in $\R^d$.}
Given $A\subset \R^d$ and $r>0$, we say that $A$ is \emph{$r$-separated} if $\norm{a-a'}\geq r$ for every $a,a'\in A, a\neq a'$. We say that $A$ is an \emph{$r$-net} if for every $x\in \R^d$ there is $a\in A$ such that $\norm{x-a}\leq r$. We call $A$ a \emph{separated net} if there are $r,s>0$ so that $A$ is an $s$-separated $r$-net.

The open and closed euclidean balls with centre $x\in\R^{d}$ and radius $r\geq 0$ are denoted by $B(x,r)$ and $\cl{B}(x,r)$ respectively. We also use the same notation for neighbourhoods of sets, i.e, $B(A,r):=\bigcup_{x\in A}B(x,r)$, where $A\subseteq\R^d$, and similarly for $\cl{B}(A,r)$.

For a set $A\subseteq\R^d$, we denote by $\partial A, \inter A$ and $\cl{A}$ the boundary, the interior and the closure of $A$ in $\R^d$, respectively.

\paragraph{Mappings.}
For a mapping $f\colon A'\subseteq\R^d\to\R^d$ such that $A\subseteq A'$, we write $\rest{f}{A}$ for the restriction of $f$ to $A$, whilst $\dom(f)$ and $\img(f)$ signify the domain and the image of $f$, respectively.

For $A\subseteq\R^{d}$ and a mapping $f\colon A\to\R^{n}$ we let
\begin{linenomath}
	\begin{equation*}
	\lip(f):=\sup\left\{\frac{\enorm{f(y)-f(x)}}{\enorm{y-x}}\colon x,y\in A, x\neq y\right\}.
	\end{equation*}
\end{linenomath}
If $f$ is injective, we let
\begin{linenomath}
	\begin{equation*}
	\bilip(f):=\max\set{\lip(f),\lip(f^{-1})}.
	\end{equation*}
\end{linenomath}
Given $L\geq 1$, we say that $f$ is $L$\emph{-Lipschitz} if $\lip(f)\leq L$ and that it is \emph{$L$-bilipschitz} if $\bilip(f)\leq L$. The mapping $f$ is said to be \emph{Lipschitz} (\emph{bilipschitz}) if there is $L<\infty$ such that $f$ is $L$-Lipschitz ($L$-bilipschitz).

\section{Bilipschitz Swapping and Gluing.}\label{section:swapping}\label{sec:proto}
The main aim of the present section is to prove the following theorem, which provides a bilipschitz mapping that efficiently performs `simultaneous swappings' of a family of specified pairs $(x,y_{x})\in\R^{d}\times \R^{d}$.
\begin{thm}\label{thm:simultaneous_switching}
	Let $X\subseteq \R^{d}$. For each $x\in X$ let $y_{x}\in\R^{d}$, $r_{x}$ be a positive real number and $\mf{U}_{x}:=B([x,y_{x}],r_{x})$.  Suppose that $(\mf{U}_{x})_{x\in X}$ is pairwise disjoint, 
	\begin{equation*}
	\sup_{x\in X,\, y_{x}\neq x}\frac{r_{x}}{\enorm{y_{x}-x}}\leq \frac{1}{2}\qquad\text{and}\qquad\sup_{x\in X,\, y_{x}\neq x}\frac{\enorm{y_{x}-x}}{r_{x}}<\infty.
	\end{equation*}
	Then there is a bilipschitz mapping $\Gamma\colon \R^{d}\to \R^{d}$ such that
	\begin{enumerate}[(i)]
		\item $\Gamma(x)=y_{x}$ and $\Gamma(y_{x})=x$ for all $x\in X$.
		\item $\Gamma(p)=p$ for all $\displaystyle p\in \R^{d}\setminus \bigcup_{x\in X, y_{x}\neq x}\mf{U}_{x}$.
		\item $\displaystyle \bilip(\Gamma)\leq \max\set{1, \sup_{x\in X,\, y_{x}\neq x}\frac{4\enorm{y_{x}-x}^{2}}{r_{x}^{2}}}$.
	\end{enumerate}
\end{thm}
The `building blocks' of the mapping $\Gamma$ from Theorem~\ref{thm:simultaneous_switching} will be bilipschitz mappings which perform a single swap of a specified pair $(x,y)\in\R^{d}\times \R^{d}$. So our first step towards the proof of Theorem~\ref{thm:simultaneous_switching} will be the construction of such mappings $\tau$ in the next Lemmas~\ref{lemma:spin} and \ref{lemma:tube_spin}. We will then need an important `gluing lemma', Lemma~\ref{lemma:simple_gluing}, to glue a family of `single swap mappings' $\tau$ together to form $\Gamma$. Lemma~\ref{lemma:simple_gluing} is also useful beyond the proof of Theorem~\ref{thm:simultaneous_switching}.
\begin{lemma}\label{lemma:spin}
	Let $d\in\N$, $d\geq 2$ and for $\theta\in \R$ let $R_{\theta}\colon \R^{d}\to\R^{d}$ be the linear mapping with matrix
	\begin{equation*}
	R_{\theta}=\begin{pmatrix}
	\cos\theta & -\sin\theta & 0 & 0 &  \ldots & 0\\
	\sin\theta & \cos\theta & 0 & \ddots& \ddots & 0\\
	0 & 0 &1 & 0 &  \ddots &0\\
	0 & 0 & 0 & 1 & \ddots& 0\\
	0 & 0 & 0 & 0 & \ddots & 0\\
	0 & 0 & 0 & 0 & 0 & 1 
	\end{pmatrix}.
	\end{equation*}
	Here we use the same notation for the linear mapping and its matrix. 
	\begin{enumerate}[(i)]
		\item\label{different_rotations} Let $\theta_{1},\theta_{2}\in\R$ and $y\in\R^{d}$. Then
		\begin{equation*}
		\enorm{R_{\theta_{2}}(y)-R_{\theta_{1}}(y)}\leq \abs{\theta_{2}-\theta_{1}}\enorm{y}.
		\end{equation*}
		\item\label{spin_Lipschitz} Let $t_{0}>0$ and $\psi\colon [0,\infty)\to \R$ be a Lipschitz mapping with $\psi(t)=0$ for all $t\geq t_{0}$. Let $\Phi=\Phi(\psi)\colon\R^{d}\to\R^{d}$ be defined by 
		\begin{equation*}
		\Phi(x)=R_{\psi(\enorm{x})}(x).
		\end{equation*}
		Then $\Phi$ is $\br{\lip(\psi)t_{0}+1}$-bilipschitz and $\Phi(x)=x$ for all $x\in \R^{d}\setminus B(0,t_{0})$.
	\end{enumerate}
\end{lemma}
\begin{proof}
	\begin{enumerate}[(i)]
		\item For $\alpha\in\R$, let $\rho_{\alpha}\colon \R^{2}\to\R^{2}$ denote the anticlockwise rotation around the origin in $\R^{2}$ through angle $\alpha$. Then, we observe that 
		\begin{equation*}
		R_{\theta}(x_{1},\ldots,x_{d})=(\rho_{\theta}(x_{1},x_{2}),x_{3},\ldots,x_{d})
		\end{equation*}
		for all $\theta\in \R$ and $x=(x_{1},\ldots,x_{d})\in\R^{d}$. Therefore,
		\begin{multline*}
		\enorm{R_{\theta_{2}}(y)-R_{\theta_{1}}(y)}=\enorm{\rho_{\theta_{2}}(y_{1},y_{2})-\rho_{\theta_{1}}(y_{1},y_{2})}\\
		\leq\abs{\theta_{2}-\theta_{1}}\enorm{(y_{1},y_{2})}\leq \abs{\theta_{2}-\theta_{1}}\enorm{y}.
		\end{multline*}
		\item First, note that $\Phi(\psi)\colon\R^{d}\to \R^{d}$ is a bijection with $\Phi(\psi)^{-1}=\Phi(-\psi)$ and $\Phi(\psi)(x)=x$ for all $x\in\R^{d}\setminus B(0,t_{0})$. Letting $\zeta\in\set{-\psi,+\psi}$ it remains to verify that $\Phi(\zeta)$ is $\br{\lip(\psi)t_{0}+1}$-Lipschitz. Since $\Phi(\zeta)$ coincides with the identity outside of $B(0,t_{0})$ it suffices to verify the Lipschitz bound between pairs of points where at least one point lies in $B(0,t_{0})$. Let $x\in\R^{d}$ and $y\in B(0,t_{0})$. Then applying part~\eqref{different_rotations} we get
		\begin{multline*}
		\enorm{\Phi(\zeta)(y)-\Phi(\zeta)(x)}\leq \enorm{R_{\zeta(\enorm{y})}(y)-R_{\zeta(\enorm{x})}(y)}+\enorm{R_{\zeta(\enorm{x})}(y-x)}\\
		\leq \abs{\zeta(\enorm{y})-\zeta(\enorm{x})}t_{0}+\enorm{y-x}\\
		\leq \lip(\zeta)\abs{\enorm{y}-\enorm{x}}t_{0}+\enorm{y-x}\leq \br{\lip(\zeta)t_{0}+1}\enorm{y-x}.\qedhere
		\end{multline*}
	\end{enumerate}	
\end{proof}

\restatespin
\begin{proof}
	The statement is invariant under scaling, so we may assume that $\enorm{y-x}=1$ and replace $r$ with $\eta:=\frac{r}{\enorm{y-x}}\leq\frac{1}{2}$. We may now also assume that $x=\frac{1}{2}e_{1}$ and $y=-\frac{1}{2}e_{1}$. We introduce a linear isomorphism $T\colon \R^{d}\to\R^{d}$ which maps the largest ellipsoid $E$ inside $B([x,y],\eta)$ centred at $0$ and with axes parallel to the coordinate axes to a ball $T(E)=B(0,R)$ of some radius $R>1$ so that $Tx=e_{1}$ and $Ty=-e_{1}$. The last condition tells us that $Te_{1}=2e_{1}$, which then implies that $R=2 \br{\frac{1}{2}+\eta}$. For $i=2,3,\ldots,d$, we then have that $T(\eta e_{i})=Re_{i}$. Thus, we get that $T$ is defined by
	\begin{equation*}
	Te_{i}=\begin{cases}
	2e_{1} & \text{if }i=1,\\
	\br{2+\frac{1}{\eta}} e_{i} & \text{if }i=2,3,\ldots,d.
	\end{cases}
	\end{equation*}
	Note that
	\begin{equation*}
	T^{-1}\br{B\br{0,1+2\eta}}\subseteq B([x,y],\eta), \qquad
	\lip(T)=2+\frac{1}{\eta},
	\qquad \lip(T^{-1})=\frac{1}{2}.
	\end{equation*}
	We now define $\psi\colon [0,\infty)\to \R$ by
	\begin{equation*}
	\psi(t)=\begin{cases}
	\pi & \text{if }0\leq t\leq 1,\\
	\frac{\br{1+2\eta-t}\pi}{2\eta} & \text{if }1\leq t\leq 1+2\eta,\\
	0 & \text{if }t\geq 1+2\eta.
	\end{cases}
	\end{equation*}
	and note that $\lip(\psi)=\frac{\pi}{2\eta}$. The desired mapping may now be defined as $\tau:=T^{-1}\circ \Phi(\psi)\circ T$, where $\Phi(\psi)$ is given by the conclusion of Lemma~\ref{lemma:spin} applied with $t_{0}=1+2\eta$, so that
	\begin{equation*}
	\bilip\br{\Phi(\psi)}\leq \frac{\pi}{2\eta}\cdot \br{1+2\eta}+1\leq \frac{\pi}{\eta}+1,
	\end{equation*}
	Multiplying $\bilip\br{\Phi(\psi)}$ by $\lip(T^{-1})\lip(T)\leq \frac{1}{\eta}$ we obtain $\bilip(\tau)\leq\frac{\pi}{\eta^{2}}+\frac{1}{\eta}\leq \frac{4}{\eta^{2}}$. With the properties of $\Phi(\psi)$ given by Lemma~\ref{lemma:spin}, property~\eqref{tau1} of $\tau$ follows from \[T\br{\R^{d}\setminus B([x,y],r)}\subseteq \R^{d}\setminus B(0,1+2\eta).\] Moreover, the calculation
	\begin{equation*}
	\tau(x)=T^{-1}\circ \Phi(\psi)\circ T(x)=T^{-1}\br{ R_{\psi(1)}(e_{1})}=T^{-1}\br{R_{\pi}(e_{1})}=T^{-1}(-e_{1})=y,
	\end{equation*}
	and a similar  one for $\tau(y)$ verify property~\eqref{tau2} of $\tau$.
\end{proof}

\restatesimpleglue
\begin{proof}
	We first show that $f$ is injective. Let $x\in\cl{A_i}$ and $y\in\cl{A_j}$ for some $i,j\in I$ such that $f(x)=f(y)$. If $i=j$, we use that $f|_{\cl{A_{i}}}$ is bilipschitz to deduce $x=y$. So assume $i\neq j$. Then, the assumptions $A_{i}\cap A_{j}=\emptyset$ and $f(\cl{A_{k}})=\cl{A_{k}}$ for $k=i,j$ imply $f(x)=f(y)\in \partial A_{i}\cap \partial A_{j}$. By Brouwer's Invariance of Domain~\cite[Thm~2B.3]{Hatcher02}, $\partial f(A_k)=f(\partial A_k)$ for every $k\in I$. Thus, $x\in\partial A_i$ and $y\in\partial A_j$. However, $\bigcup_{i\in I}\partial A_{i}\subseteq \R^{d}\setminus \bigcup_{i\in I}A_{i}$ and $f|_{\R^{d}\setminus \bigcup_{i\in I}A_{i}}$ is bilipschitz, so, again, we must have $x=y$. Hence, $f|_{\bigcup_{i\in I}\cl{A_{i}}}$ is injective. By hypothesis $f|_{\R^{d}\setminus \bigcup_{i\in I}A_{i}}$ is also injective and the images of $f|_{\bigcup_{i\in I}A_{i}}$ and $f|_{\R^{d}\setminus \bigcup_{i\in I}A_{i}}$ do not intersect.
	Consequently, $f$ is injective and $f^{-1}$ is well-defined.
	
	Now observe that the assumptions of the lemma are symmetric with respect to interchanging $f$ and $f^{-1}$. Therefore, it suffices to show that $f$ is $C$-Lipschitz.
	
	By the assumptions of the lemma, the $C$-Lipschitz inequality is already satisfied for pairs $x,y\in\R^{d}$ with $x,y\in\cl{A_{i}}$ for some $i\in I$, or $x,y\in\R^{d}\setminus \bigcup_{i\in I}A_{i}$. We verify the inequality for the remaining types of pairs $x,y\in\R^{d}$. 
	
	Let $i,j\in I, i\neq j$, and $x\in A_i\setminus\cl{A_j}$ and $y\in A_j\setminus\cl{A_i}$. We set $z_1:=x$, $z_4:=y$ and take $z_2$ as any point of $[x,y]\cap\partial A_i$. Since $x\in A_i$ and $y\notin\cl{A_i}$, $z_2$ is well-defined. Note that $\partial A_{i}\cap A_{j}=\emptyset$, so $z_{2}\notin A_{j}$. Finally, we take $z_3$ as any point of $[z_2, y]\cap\partial A_j$.
	Therefore, by the assumptions on the $C$-bilipschitz property,
	\begin{equation*}
	\enorm{f(y)-f(x)}\leq \sum_{k=2}^{4}\enorm{f(z_{k})-f(z_{k-1})}\leq C\sum_{k=2}^{4}\enorm{z_{k}-z_{k-1}}=C\enorm{y-x}.
	\end{equation*}
	The last remaining case to consider is $x\in\R^{d}\setminus \bigcup_{i\in I}A_{i}$ and $y\in A_{j}$ for some $j\in I$. Then a similar argument as above applies with $z_{1}:=x$, $z_{3}:=y$ and $z_{2}$ as any point of $[x,y]\cap\partial A_{j}$.
\end{proof}

We now have all the necessary tools to prove Theorem~\ref{thm:simultaneous_switching}.
\begin{proof}[Proof of Theorem~\ref{thm:simultaneous_switching}]
	For each $x\in X$ with $y_{x}=x$ let $\tau_{x}\colon \R^{d}\to\R^{d}$ be the identity mapping. For each remaining $x\in X$ let $\tau_{x}\colon\R^{d}\to \R^{d}$ be the mapping given by Lemma~\ref{lemma:tube_spin} applied to $x$, $y_{x}$ and $r_{x}$. Setting $C:=\displaystyle\max\set{1, \sup_{x\in X,\, y_{x}\neq x}\frac{4\enorm{y_{x}-x}^{2}}{r_{x}^{2}}}$, note that each $\tau_x$ is $C$-bilipschitz and equal to the identity outside $\mf{U}_x$.  We then define $\Gamma\colon \R^{d}\to \R^{d}$ by 
	\begin{equation*}
	\Gamma(p)=\begin{cases}
	\tau_{x}(p) & \text{if $x\in X$ and $p\in \mf{U}_{x}$},\\
	p & \text{if }p\in \R^{d}\setminus \bigcup_{x\in X}\mf{U}_{x}.
	\end{cases}
	\end{equation*}
	Finally, we may apply Lemma~\ref{lemma:simple_gluing} to the collection of open sets $(\mf{U}_{x})_{x\in X}$ and the mapping $\Gamma\colon \R^{d}\to \R^{d}$ to conclude that $\Gamma$ is $C$-bilipschitz.
\end{proof}

\section{Reduction to the integer lattice.}\label{s:generalise}
This section is dedicated to the full proof of Theorem~\ref{thm:Zd_equiv}, in which we reduce the bilipschitz extension problem for all separated nets to that for the integer lattice. Before proceeding, the reader may wish to read the sketch proof of Theorem~\ref{thm:Zd_equiv} in the introduction.
\begin{lemma}\label{lemma:rounding_to_integers}
	Let $d\in\N$, $d\geq 2$, $r>0$ and $A\subseteq \R^{d}$ be $r$-separated. Then there exists a $16\max\set{\frac{3d}{r},1}$-bilipschitz mapping $\Phi\colon \R^{d}\to\R^{d}$ such that $\Phi(A)\subseteq \Z^{d}$.
\end{lemma}
\begin{proof}
	If $r\geq3d$ let $\Psi$ stand for the identity mapping of $\R^{d}$. Otherwise let $\Psi\colon\R^{d}\to\R^{d}$ be the rescaling defined by $\Psi(x)=\frac{3dx}{r}$ for all $x\in\R^{d}$. Then $\Psi(A)$ is $3d$-separated. For each $x\in A$ we choose $y_{x}\in \Z^{d}$ such that $\enorm{y_{x}-\Psi(x)}\leq\sqrt{d}/2$. Then for $x,x'\in A$ with $x\neq x'$ we have $\enorm{y_{x}-y_{x'}}\geq \enorm{\Psi(x)-\Psi(x')}-\sqrt{d}\geq 3d-\sqrt{d}> 2\sqrt{d}$. Therefore the collection of balls $(B(y_{x},\sqrt{d}))_{x\in A}$ is pairwise disjoint. For each $x\in X$ let 
	\begin{equation*}
	r_{x}:=\begin{cases}
	\enorm{y_{x}-\Psi(x)}/2 & \text{if }y_{x}\neq \Psi(x),\\
	\sqrt{d} & \text{otherwise.}
	\end{cases}
	\end{equation*}
	and observe that $\mf{U}_{x}:=B([\Psi(x),y_{x}],r_{x})\subseteq B(y,\sqrt{d})$. Hence $(\mf{U}_{x})_{x\in A}$ is pairwise disjoint.
	Let $\Pi\colon\R^{d}\to\R^{d}$ be the 16-bilipschitz mapping given by the application of Theorem~\ref{thm:simultaneous_switching} to $X=\Psi(A)$, $\Psi(x)$ (in place of $x$), $y_{x}$ and $r_{x}$ for each $x\in A$.
	All that remains is to set $\Phi=\Pi\circ \Psi$.
\end{proof}

\begin{lemma}~\label{lemma:ext_to_lattice}
	Let $d\in\N$, $d\geq 2$, $L\geq 1$, $X\subseteq \Z^{d}$ and $f\colon X\to \R^{d}$ be an $L$-bilipschitz mapping. Suppose that there exists $\lambda\geq 1$ such that $\R^{d}\subseteq \bigcup_{x\in X}\cl{B}(x,\lambda)$. Then there exists a bilipschitz extension $F\colon \Z^{d}\to\R^{d}$ of $f$ with
	\begin{equation*}
	\lip(F)\leq 4\lambda L, \qquad \lip(F^{-1})\leq 24\lambda^{2}L\sqrt{d}.
	\end{equation*}
\end{lemma}
\begin{proof}
	For each $y\in\Z^{d}\setminus X$ choose $\alpha_{y}\in X$ such that $y\in \cl{B}(\alpha_{y},\lambda)$. Set $\alpha_{x}=x$ for every $x\in X$.
	For any $t\geq\sqrt{d}$, comparing a euclidean ball $\cl{B}(0,t)$ with cubes $[-t, t]^d$ and $\sqbr{-\frac{t}{\sqrt{d}},\frac{t}{\sqrt{d}}}^d$ we get the following bounds:
	\begin{equation}\label{eq:ball_lattice}
	2^d\br{\frac{t}{\sqrt{d}}-1}^d\leq\abs{\cl{B}(0,t)\cap \Z^{d}}\leq 2^d(t+1)^d
	\end{equation}
	In particular, for each $x\in X$ we have
	\begin{equation*}
	\abs{\set{y\in\Z^{d}\colon \alpha_{y}=x}}\leq \abs{\cl{B}(0,\lambda)\cap \Z^{d}}\leq 2^d(\lambda+1)^{d}.
	\end{equation*}
	Next, observe that the balls $(B(f(x),\frac{1}{2L}))_{x\in X}$ are pairwise disjoint. Setting $N:=12\lambda\sqrt{d}L$ and using \eqref{eq:ball_lattice} we infer that for every $w\in\R^d$
	\[
	\abs{\br{\cl{B}\br{w,\frac{1}{4L}}}\cap \br{w+\frac{1}{N}\Z^{d}}}=\abs{\cl{B}\br{0,\frac{N}{4L}}\cap \Z^{d}}\geq 2^d(\lambda+1)^d.
	\]
	As $x\in\set{y\in \Z^{d}\colon \alpha_{y}=x}$ for all $x\in X$, it follows that there is an injective mapping $F\colon \Z^{d}\to \R^d$ such that 
	\begin{equation*}\label{eq:iota_lattice}
	F(y)\in \br{\cl{B}\br{f(\alpha_{y}),\frac{1}{4L}}}\cap \br{f(\alpha_{y})+\frac{1}{N}\Z^{d}}\qquad\text{ for every $y\in \Z^{d}$}
	\end{equation*}
	and $F(x)=f(x)$ for all $x\in X$.
	These conditions imply $\enorm{F(p)-F(\alpha_p)}\leq 1/(4L)$ for every $p\in\Z^d$.
	
	Let $p,q\in \Z^{d}$ with $p\neq q$. We verify the bilipschitz bounds on $\enorm{F(q)-F(p)}$. The definition of $(\alpha_{y})_{y\in\Z^{d}}$ implies the following inequality, which we will use several times:
	\begin{equation*}
	\abs{\enorm{\alpha_{q}-\alpha_{p}}-\enorm{q-p}}\leq 2\lambda.
	\end{equation*}
	Since $\enorm{q-p}\geq 1$ we have
	\begin{equation*}
	\enorm{F(q)-F(p)}\leq \enorm{F(\alpha_{q})-F(\alpha_{p})}+\frac{1}{2L}\leq \br{L+2L\lambda+\frac{1}{2L}}\enorm{q-p}.
	\end{equation*}
	The lower bilipschitz bound is easy in the case that $\alpha_{p}\neq \alpha_{q}$: then it holds that $\enorm{F(q)-F(p)}\geq \enorm{F(\alpha_{q})-F(\alpha_{p})}-\frac{1}{2L}\geq \frac{1}{2L}$. So assume now that $\alpha_{p}=\alpha_{q}$. Then $\enorm{q-p}\leq 2\lambda$ and $\enorm{F(q)-F(p)}\geq 1/N$, so $\enorm{F(q)-F(p)}\geq \frac{1}{2\lambda N}\enorm{q-p}$. This completes the verification of the bilipschitz property for $F$ and substituting in the set value $N=12\lambda \sqrt{d}L$ to the estimates above delivers the desired bounds on $\lip(F)$ and $\lip(F^{-1})$.
\end{proof}

\begin{proof}[Proof of Theorem~\ref{thm:Zd_equiv}]
	For the equivalence, it suffices to prove the implication \eqref{Zd}$\Rightarrow$\eqref{gen_sep_net}, as the other one is trivial. Assume that \eqref{Zd} holds, let $A\subset\R^d$ be an $r$-separated $R$-net and let $f\colon A\to \R^{d}$ be bilipschitz.
	We set $K:=16\max\set{\frac{3d}{r},1}$ and apply Lemma~\ref{lemma:rounding_to_integers} to $A$ to get a $K$-bilipschitz mapping $\Phi\colon\R^d\to\R^d$ such that $\Phi(A)\subset\Z^d$.
	Next, we apply Lemma~\ref{lemma:ext_to_lattice} to $\rest{f\circ \Phi^{-1}}{\Phi(A)}$ and obtain its bilipschitz extension $g\colon \Z^d\to\R^d$ with $\bilip(g)\leq 24\sqrt{d} R^2K^3\bilip(f)$.
	Finally, \eqref{Zd} provides us with a bilipschitz extension $G\colon \R^{d}\to \R^{d}$ of $g$. The desired extension of $f$ is then given by $F:=G\circ \Phi$.
	If \eqref{Zd} holds with $\bilip(F)\leq C_d(\bilip(f))$ for some monotone function $C_d\colon [1,\infty)\to [1,\infty)$ then the bilipschitz extension $F$ of $f\colon A\to \R^{d}$ constructed above to verify \eqref{gen_sep_net} has bilipschitz constant bounded above by $\bilip(\Phi)\bilip(G)\leq K\cdot C_d\br{24\sqrt{d}R^{2}K^{3}\bilip(f)}$.
\end{proof}

\section{Threading a bilipschitz cord through the channel.}\label{s:threading}
In the present section we extend a given bilipschitz mapping of two `horizontal' hyperplanes in $\R^{d}$ and all integer lattice points between them, under the additional conditions that the mapping doesn't break obvious topological requirements for extendability and coincides with the identity on the two hyperplanes. The section title comes from the early stages of the development of this work, where we took the dimension $d=2$ and the two horizontal hyperplanes, or lines, to have heights $\pm\frac{1}{2}$. In this special case there is only `one layer', $\Z\times\set{0}$, of integer lattice points between the two lines and we imagined our task of extending $f$ to be like threading a bilipschitz cord through the points $f(x,0)$ for $x\in\Z$ in the correct order to determine the extension on $\R\times\set{0}$.

\restatethread

The reader may wish to consult the sketch proof of Theorem~\ref{thm:thread} in the introduction before proceeding.

\begin{lemma}\label{lemma:inj_rounding}
	Let $d,H\in\N$, $d\geq2$, $0<s\leq\frac{1}{4}$, $N\geq(2\sqrt{d}/s)^3H^{\frac{1}{d-1}}$,
	\begin{equation*}
	X_{1},X_{2},\ldots,X_{H}\subseteq \R^{d-1}\times \sqbr{\frac{1}{2}+s,H+\frac{1}{2}-s} 
	\end{equation*}
	be pairwise disjoint and suppose that $X:=\displaystyle\bigcup_{m=1}^{H}X_{m}$ is $s$-separated. 
	Then there exists a $\displaystyle 2^{8}N^{2}H^{2}$-bilipschitz mapping $\Psi\colon \R^{d}\to\R^{d}$ such that
	\begin{enumerate}[(i)]
		\item\label{Xi_i} $\Psi(x)=x$ for all $x\in\R^{d-1}\times\br{\R\setminus \br{\frac{1}{2},H+\frac{1}{2}}}$.
		\item\label{Xi_ii} $\Psi(X_{m})\subseteq \br{\frac{1}{N}\Z^{d-1}\setminus \Z^{d-1}}\times\set{m}$ for $m=1,2,\ldots,H$.
		\item\label{Xi_iii} $\enorm{\proj_{\R^{d-1}}\circ \Psi(x)-\proj_{\R^{d-1}}(x)}\leq s^{2}$ for all $x\in X$.
	\end{enumerate}
\end{lemma}
\begin{proof}
	We will construct $\Psi$ as a composition $\Omega\circ \Xi$, where $\Xi$ will take care of the first $(d-1)$-coordinates and $\Omega$ of the last coordinate. 
	\paragraph{Construction of $\Xi$.}
	We will define a $16$-bilipschitz mapping $\Xi\colon\R^{d}\to\R^{d}$ with the following properties
	\begin{enumerate}[(A)]
		\item\label{Psi_i} $\Xi(x)=x$ for all $x\in\R^{d-1}\times \br{\R\setminus \br{\frac{1}{2},H+\frac{1}{2}}}$,
		\item\label{Psi_ii} $\enorm{\Xi(x)-x}\leq s^{2}$ for all $x\in X$. 
		\item\label{Psi_iii} $\proj_{\R^{d-1}}\circ \Xi(X)\subseteq \frac{1}{N}\Z^{d-1}\setminus \Z^{d-1}$ and $\proj_{\R^{d-1}}\circ \Xi|_{X}$ is injective.
		\item\label{Psi_iv} $\proj_{d}\circ \Xi(x)=\proj_{d}(x)$ for all $x\in X$.
	\end{enumerate}
	Let $Q:=\left[0,\frac{s^{2}}{\sqrt{d-1}}\right)^{d-1}$ and for each $x\in X$ let $\Phi(x)\in\Z^{d-1}$ be defined by the condition
	\begin{equation*}
	\proj_{\R^{d-1}}(x)\in Q+\frac{s^{2}}{\sqrt{d-1}}\Phi(x).
	\end{equation*}
	For each $z\in\Z^{d-1}$ we have that $\bigcup_{x\in\Phi^{-1}(\set{z})}B(x,s/2)$ is a pairwise disjoint union of balls all contained in the set $\br{\cl{B}_{\R^{d-1}}(Q,s/2)+\frac{s^{2}z}{\sqrt{d-1}}}\times\sqbr{\frac{1}{2},H+\frac{1}{2}}$. Hence, by comparing volumes and writing $\alpha_{d}$ for the $d$-dimensional volume of the unit ball in $\R^{d}$, we get
	\begin{equation}\label{eq:Phi_preim_upper_bound}
	\abs{\Phi^{-1}(\set{z})}\leq \frac{(s+s^{2})^{d-1}H}{\alpha_{d}\br{s/2}^{d}}\leq\frac{(2s)^{d-1}2^{d}H}{\alpha_{d}s^{d}}\leq \frac{2^{d-1}\sqrt{d^{d}}H}{s}.
	\end{equation}
	On the other hand we have that 
	\begin{equation}\label{eq:tricky}
	\abs{\br{Q+\frac{s^{2}}{\sqrt{d-1}}z}\cap \br{\frac{1}{N}\Z^{d-1}\setminus\Z^{d-1}}}\geq \br{\frac{s^{2}/\sqrt{d}}{1/N}-1}^{d-1}-1\geq \frac{1}{2}\br{\frac{Ns^{2}}{2\sqrt{d}}}^{d-1} .
	\end{equation}
	The lower bound on $N$ in the hypothesis is used for the last inequality. It also ensures that the final quantity of \eqref{eq:tricky} is at least that of \eqref{eq:Phi_preim_upper_bound}. Therefore, there exists for each $z\in\Z^{d-1}$ an injective mapping  
	\begin{equation*}
	\iota_{z}\colon \Phi^{-1}(\set{z})\to \br{Q+\frac{s^{2}}{\sqrt{d-1}}z}\cap \br{\frac{1}{N}\Z^{d-1}\setminus \Z^{d-1}}.
	\end{equation*}
	For each $x\in X$ and a radius $r_{x}=r_{x}(s)$ to be determined later, we let 
	\begin{equation*}
	y_{x}:=(\iota_{\Phi(x)}(x),x_{d})\in \R^{d}, \qquad \mf{U}_{x}:=B([x,y_{x}],r_{x}).
	\end{equation*}
	Observe that for each $x\in X$ the line segment $[x,y_{x}]$ is parallel to the hyperplane $\R^{d-1}\times\set{0}$ and has length 
	\begin{equation}\label{eq:dist_x-y_x}
	\enorm{y_{x}-x}\leq \diam\br{Q}=s^{2}.
	\end{equation}
	Therefore whenever $x,x'\in X$ and $\proj_{\R^{d-1}}\mf{U}_{x}\cap \proj_{\R^{d-1}}{\mf{U}_{x'}}\neq \emptyset$ we have 
	\begin{equation*}
	\enorm{\proj_{\R^{d-1}}(x')-\proj_{\R^{d-1}}(x)}\leq 2s^{2}+ r_{x}+r_{x'}\leq 3s^{2},
	\end{equation*}
	where the last inequality comes from imposing the condition $r_{x}(s)\leq s^{2}/2$ for every $x\in X$. Therefore, whenever $x$ and $x'$ are distinct points of the $s$-separated set $X$ and $\proj_{\R^{d-1}}\mf{U}_{x}\cap \proj_{\R^{d-1}}{\mf{U}_{x'}}\neq \emptyset$ we have
	\begin{equation*}
	\abs{\proj_{d}(x')-\proj_{d}(x)}\geq s-3s^{2}\geq s^{2}\geq 2\sup_{w\in X}r_{w}(s).
	\end{equation*}
	We deduce that the collection $\br{\mf{U}_{x}}_{x\in X}$ is pairwise disjoint. Moreover, the condition $X\subseteq \R^{d-1}\times \sqbr{\frac{1}{2}+s,H+\frac{1}{2}-s}$ and the fact that each $[x,y_{x}]$ is parallel to $\R^{d-1}\times \set{0}$ ensure that each set $\mf{U}_{x}$ in this collection is contained in $\R^{d-1}\times \br{\frac{1}{2},H+\frac{1}{2}}$.

	Let $\Xi\colon \R^{d}\to\R^{d}$ be the $16$-bilipschitz mapping given by the application of Theorem~\ref{thm:simultaneous_switching} to $X$, $(y_{x})_{x\in X}$ and $(r_{x})_{x\in X}$.
	Here we need one last condition on $r_{x}(s)$, namely $r_{x}(s)\leq \enorm{y_{x}-x}/2$ for all $x\in X$ with $y_{x}\neq x$; the choice
	\begin{equation*}
	r_{x}(s):=\begin{cases}
	\frac{1}{2}\enorm{y_{x}-x} & \text{if $y_{x}\neq x$,}\\
	\frac{1}{2}s^{2} & \text{otherwise}
	\end{cases}
	\end{equation*}
	satisfies both conditions we required. The properties \eqref{Psi_i}--\eqref{Psi_iv} of $\Xi$ are now easily verified.
	
	\paragraph{Construction of $\Omega$.} 
	We construct a $16N^{2}H^{2}$-bilipschitz mapping $\Omega$ with the following properties:
	\begin{enumerate}[(I)]
		\item\label{rho_i} $\Omega(x)=x$ for all $x\in\R^{d-1}\times \br{\R\setminus \br{\frac{1}{2},H+\frac{1}{2}}}$.
		\item\label{rho_ii} $\proj_{\R^{d-1}}\circ \Omega\circ \Xi(x)=\proj_{\R^{d-1}}\circ \Xi(x)$ for all $x\in X$.
		\item\label{rho_iii} $\proj_{d}\circ \Omega\circ \Xi(x)=m$ for all $x\in X_{m}$ and $m=1,2,\ldots,H$.
	\end{enumerate}
	For each $m\in [H]$ and $x\in X_{m}$ let 
	\begin{align*}
	v_{x}&:=\Xi(x), & y_{x}:=(\proj_{\R^{d-1}}\circ \Xi(x),m)\in \R^{d},\\
	r_{x}&:=\begin{cases}
	\min\set{\frac{1}{2N},\frac{\enorm{y_{x}-v_{x}}}{2}} & \text{if }v_x\neq y_x,\\
	\frac{1}{2N} & \text{if }v_x=y_x,
	\end{cases} & \mf{U}_{x}:=B([v_{x},y_{x}], r_{x}).
	\end{align*}
	The choice of these objects, together with property \eqref{Psi_iii} of $\Xi$, ensures that the collection $(\mf{U}_{x})_{x\in X}$ is pairwise disjoint. Moreover, the hypotheses on $X$, the bounds $r_{x}<1/N<s$ and \eqref{Psi_iv} ensure that each set $\mf{U}_{x}$ with $x\in X$ is contained in $\R^{d-1}\times \br{\frac{1}{2},H+\frac{1}{2}}$.

	Let the $16N^{2}H^{2}$-bilipschitz mapping $\Omega\colon \R^{d}\to \R^{d}$ be given by Theorem~\ref{thm:simultaneous_switching} applied to $\set{v_{x}\colon x\in X}$ in place of $X$, $v_{x}$ in the role of $x$, $(y_{x})_{x\in X}$ and $(r_{x})_{x\in X}$. The properties \eqref{rho_i}--\eqref{rho_iii} of $\Omega$ are clear from the construction. 
	\paragraph{Properties of $\Psi:=\Omega\circ \Xi$.}
	Setting $\Psi:=\Omega\circ \Xi$ we get that $\Psi$ is $2^{8}N^{2}H^{2}$-bilipschitz. The properties \eqref{Xi_i}--\eqref{Xi_iii} of $\Psi$ follow easily from \eqref{Psi_i}--\eqref{Psi_iv} for $\Xi$ and \eqref{rho_i}--\eqref{rho_iii} for $\Omega$.
\end{proof}

\begin{lemma}\label{lemma:_specific_applied_permute}
	Let $d\geq 2$, $T,N\in \N$, $m\in\R$ and $\sigma\colon \frac{1}{N}\Z^{d-1}\to \frac{1}{N}\Z^{d-1}$ be a permutation with $\enorm{\sigma(x)-x}\leq T$ for all $x\in \frac{1}{N}\Z^{d-1}$. Then there exists a
	\begin{align*}
	2^{6d+10}N^{2d+1}T^{2d}&\text{-bilipschitz mapping}\\
	\Upsilon\colon\R^{d-1}\times\sqbr{m-\frac{1}{2},m+\frac{1}{2}}&\to \R^{d-1}\times \sqbr{m-\frac{1}{2},m+\frac{1}{2}}
	\end{align*}
	such that 
	\begin{enumerate}[(i)]
		\item $\Upsilon(x)=x$ for all $x\in\R^{d-1}\times \set{m-\frac{1}{2},m+\frac{1}{2}}$.
		\item $\Upsilon(x,m)=(\sigma(x),m)$ for all $x\in \frac{1}{N}\Z^{d-1}$.
	\end{enumerate}
\end{lemma}
\begin{proof}
	Suppose that the lemma is valid whenever $N=1$ and $m=0$. Then we may consider the permutation $\wt{\sigma}\colon \Z^{d-1}\to\Z^{d-1}$ defined by $\wt{\sigma}(x)=N\sigma\br{\frac{x}{N}}$ and note that $\enorm{\wt{\sigma}(x)-x}\leq NT$ for all $x\in \Z^{d-1}$. Let $\wt{\Upsilon}\colon \R^{d-1}\times\sqbr{-\frac{1}{2},\frac{1}{2}}\to \R^{d-1}\times\sqbr{-\frac{1}{2},\frac{1}{2}}$ be the $2^{6d+10}\br{NT}^{2d}$-bilipschitz mapping given by the version of the theorem with $N=1$ and $m=0$. The full statement is now verified by the mapping $\Upsilon\colon \R^{d-1}\times\sqbr{m-\frac{1}{2},m+\frac{1}{2}}\to\R^{d-1}\times\sqbr{m-\frac{1}{2},m+\frac{1}{2}}$ defined by 
	\begin{equation*}
	\Upsilon(x)=\rho^{-1}\circ\wt{\Upsilon}\circ \rho(x)\qquad\text{for all $x\in \R^{d-1}\times\sqbr{m-\frac{1}{2},m+\frac{1}{2}}$,}
	\end{equation*}
	where $\rho\colon \R^{d-1}\times \sqbr{m-\frac{1}{2},m+\frac{1}{2}}\to\R^{d-1}\times\sqbr{-\frac{1}{2},\frac{1}{2}}$ is the affine bijection given by $\rho(x,h)=(Nx,h-m)$ for all $x\in\R^{d-1}$ and $h\in\sqbr{m-\frac{1}{2},m+\frac{1}{2}}$.
	
	It remains to prove the lemma in the case $N=1$ and $m=0$. In the proof that follows we will work with Euclidean balls around sets, both in $\R^{d-1}$ and in $\R^{d}$. Therefore, we will extend our usual ball notation to $B_{\R^{d-1}}$ and $B_{\R^{d}}$, to distinguish between them. Recall that these notations always refer to open balls.
	
	Define an infinite graph $G$ on the vertex set $\Z^{d-1}$ by prescribing that two points $x,z\in\Z^{d-1}$ are connected by an edge if and only if $x\neq z$ and $B_{\R^{d-1}}([x,\sigma(x)],1)\cap B_{\R^{d-1}}([z,\sigma(z)],1)\neq \emptyset$. Note that the degree of any vertex of $G$ is uniformly bounded. In fact, whenever $x,z\in\Z^{d-1}$ are connected by an edge in $G$ we have that
	\begin{equation*}
	\enorm{z-x}< \enorm{z-\sigma(z)}+\enorm{\sigma(x)-x}+2\leq 2T+2\leq 4T.
	\end{equation*}
	Hence, any vertex $x\in\Z^{d-1}$ can have at most $(8T)^{d-1}-1$ adjacent vertices $z\in\Z^{d-1}$: To verify this, note that any adjacent vertex $z$ of $x$ satisfies 
	\begin{equation*}
		z-x\in\set{-(4T-1),-(4T-2),\ldots,0,1,\ldots,4T-1}^{d-1}\setminus \set{0}^{d-1}.
	\end{equation*}
	Using the standard graph theoretic notation $\Delta$ for maximal degree of any vertex, we have thus verified $\Delta(G)\leq 2^{3(d-1)}T^{d-1}-1$. Let $\iota\colon \Z^{d-1}\to\sqbr{2^{3(d-1)}T^{d-1}}$ be a proper colouring of the vertices of $G$, that is, a mapping of the set of vertices $\Z^{d-1}$ which assigns different values, or `colours', to any pair of adjacent vertices. Here we are making use of the trivial fact that any countable graph $\Gamma$ with $\Delta(\Gamma)<\infty$ admits a proper colouring with $\Delta(\Gamma)+1$ colours. Such a proper colouring can be obtained via a greedy algorithm~\cite[Proposition 5.2.2]{DiestelGT}.
	
	For each $x\in \Z^{d-1}$ let $y_{x}\upp{i}\in\R^{d}$, $r_{x}\upp{i}>0$ and $\mf{U}_{x}\upp{i}\subseteq \R^{d}$ for $i\in[3]\cup\set{0}$ or $i\in[3]$ be defined as 
	\begin{align*}
	y_{x}\upp{0}&:=(x,0),\,\, y_{x}\upp{1}:=\br{x,\frac{\iota(x)}{2^{3d-1}T^{d-1}}},\,\, y_{x}\upp{2}:=\br{\sigma(x),\frac{\iota(x)}{2^{3d-1}T^{d-1}}},\,\, y_{x}\upp{3}:=\br{\sigma(x),0},\\
	r_{x}\upp{i}&:=\frac{\enorm{y_{x}\upp{i}-y_{x}\upp{i-1}}}{2}\quad\text{for $i=1,3$,}\qquad \text{and}\qquad r_{x}\upp{2}:=\frac{1}{2^{3d}T^{d-1}},\\
	\mf{U}_{x}\upp{i}&:=B_{\R^{d}}\br{\sqbr{y_{x}\upp{i-1},y_{x}\upp{i}},r_{x}\upp{i}} \qquad \text{for each $i\in[3]$.}
	\end{align*}
	Note that $\mf{U}_{x}\upp{i}\subseteq \R^{d-1}\times\br{-\frac{1}{2},\frac{1}{2}}$ for each $x\in\Z^{d-1}$ and $i\in[3]$.
	
	Our aim, for each $i\in [3]$, is to apply Theorem~\ref{thm:simultaneous_switching} to the collection of sets $(\mf{U}_{x}\upp{i})_{x\in\Z^{d-1}}$. Therefore, the next part of the proof will be focussed on establishing the conditions of Theorem~\ref{thm:simultaneous_switching}. 
	
	We claim, for each $i\in[3]$, that the collection of sets $(\mf{U}_{x}\upp{i})_{x\in \Z^{d-1}}$ is pairwise disjoint. This is obviously true for $i=1,3$, since the segments $\sqbr{y_{x}\upp{i-1},y_{x}\upp{i}}$ are are all parallel to $e_{d}$ and the set $\set{y_{x}\upp{i-1}\colon x\in\Z^{d-1}}$ is $1$-separated. For $i=2$, our claim relies on the proper colouring $\iota$ of the graph $G$. For distinct points $x,z\in\Z^{d-1}$ we verify that $\mf{U}_{x}\upp{2}\cap \mf{U}_{z}\upp{2}=\emptyset$ by distinguishing two cases. If $B_{\R^{d-1}}([x,\sigma(x)],1)\cap B_{\R^{d-1}}([z,\sigma(z)],1)=\emptyset$ then $B_{\R^{d}}\br{\sqbr{y_{x}\upp{1},y_{x}\upp{2}},1}\cap B_{\R^{d}}\br{\sqbr{y_{z}\upp{1},y_{z}\upp{2}},1}=\emptyset$ and the claim follows. In the remaining case we have that $\iota(x)\neq \iota(z)$. Hence the line segments $\sqbr{y_{x}\upp{1},y_{x}\upp{2}}$ and $\sqbr{y_{z}\upp{1},y_{z}\upp{2}}$ are both orthogonal to $e_{d}$ and their last coordinate projections differ by at least $\frac{1}{2^{3d-1}T^{d-1}}$. It follows that $\mf{U}_{x}\upp{2}\cap \mf{U}_{z}\upp{2}=\emptyset$.  
	
	To establish the remaining conditions of Theorem~\ref{thm:simultaneous_switching}, we observe the following:
	\begin{align*}
	r_{x}\upp{i}&\leq \frac{1}{2}\enorm{y_{x}\upp{i}-y_{x}\upp{i-1}}\quad\text{if $i\in [3]$, $x\in\Z^{d-1}$ and $y_{x}\upp{i}\neq y_{x}\upp{i-1}$,}\\
	\frac{\enorm{y_{x}\upp{i}-y_{x}\upp{i-1}}}{r_{x}\upp{i}}&=2 \qquad \text{for $i=1,3$ and $x\in \Z^{d-1}$, and}\\
	\frac{\enorm{y_{x}\upp{2}-y_{x}\upp{1}}}{r_{x}\upp{2}}&\leq\frac{T}{\frac{1}{2^{3d}T^{d-1}}}=2^{3d}T^{d}\quad\text{if $x\in\Z^{d-1}$ and $y_{x}\upp{2}\neq y_{x}\upp{1}$.}
	\end{align*}
	For each $i\in [3]$ let $\Gamma_{i}\colon\R^{d}\to\R^{d}$ be the bilipschitz mapping given by Theorem~\ref{thm:simultaneous_switching} applied with $X=\set{y_{x}\upp{i-1}\colon x\in\Z^{d-1}}$, $y_{x}\upp{i-1}$ in place of $x$, $y_{x}\upp{i}$ in place of $y_{x}$ and $r_{x}\upp{i}$ in place of $r_{x}$. Observe the following properties of $\Gamma_{i}$ for $i\in[3]$:
	\begin{enumerate}[(i)]
		\item\label{Gammai1} $\Gamma_{i}\br{y_{x}\upp{i-1}}=y_{x}\upp{i}$ for each $x\in \Z^{d-1}$.
		\item\label{Gammai3} $\Gamma_{i}(x)=x$ for all $x\in \R^{d-1}\times\br{\R\setminus\br{-\frac{1}{2},\frac{1}{2}}}$.
		\item\label{Gammai4} $\bilip(\Gamma_{1})\leq 16$, $\bilip(\Gamma_{2})\leq 2^{6d+2}T^{2d}$, and $\bilip(\Gamma_{3})\leq 16$.
	\end{enumerate}
	The desired mapping may now be defined by $\Upsilon:=\Gamma_{3}\circ\Gamma_{2}\circ\Gamma_{1}$.
\end{proof}

We are now ready to prove Theorem~\ref{thm:thread}:
\begin{proof}[Proof of Theorem~\ref{thm:thread}]
	Let the
	\begin{equation*}
	2^{8}N^{2}H^{2}\text{-bilipschitz mapping } \Psi\colon \R^{d}\to\R^{d}
	\end{equation*}
	be given by the conclusion of Lemma~\ref{lemma:inj_rounding} applied to 
	\begin{equation}\label{eq:setup}
	s:=\frac{1}{4L},\quad X_{m}:=f\br{\Z^{d-1}\times\set{m}},\quad  X:=f(\Z^{d-1}\times[H]), \quad N:=\floor{2(2\sqrt{d}(4L))^{3}H^{\frac{1}{d-1}}}\in\N.
	\end{equation}
	We apply the obvious upper bound for $N$ to record that
	\begin{equation}\label{eq:bPsi}
	\bilip(\Psi)\leq 2^{8}N^{2}H^{2}\leq 2^{28}d^{3}L^{6}H^{2+\frac{2}{d-1}}\leq 2^{14d}d^{3d/2}L^{3d}H^{2d}.
	\end{equation}
	Let $m\in [H]$. Then we define 
	\begin{equation*}
	\xi_{m}\colon \Z^{d-1}\to \frac{1}{N}\Z^{d-1}\setminus \Z^{d-1}
	\end{equation*}
	by
	\begin{equation*}
	\xi_{m}(x)=\proj_{\R^{d-1}}\circ \Psi\circ f(x,m),\qquad x\in \Z^{d-1}.
	\end{equation*}
	Using properties \eqref{Xi_i} and \eqref{Xi_ii} of $\Psi$ from Lemma~\ref{lemma:inj_rounding}, we have
	\begin{multline}\label{eq:xi-id}
	\enorm{\xi_{m}(x)-x}= \enorm{\Psi\circ f(x,m)-(x,m)}\\
	\leq \enorm{\Psi\circ f(x,m)-\Psi\circ f\biggl(x,\frac{1}{2}\biggr)}+\enorm{\biggl(x,\frac{1}{2}\biggr)-(x,m)}\\
	\leq 2^{8}LN^{2}H^{3}+H\leq \floor{2^{9}LN^{2}H^{3}}=:T, \qquad \text{for all $x\in \Z^{d-1}$.}
	\end{multline}
	Since $N\in\N$ we have that $\Z^{d-1}\subseteq \frac{1}{N}\Z^{d-1}$, so, in particular, $\xi_{m}$ is defined on a subset of $\frac{1}{N}\Z^{d-1}$. We now extend $\xi_{m}$ to a permutation $\sigma_{m}\colon \frac{1}{N}\Z^{d-1}\to \frac{1}{N}\Z^{d-1}$ as follows: Since each $\xi_{m}$ is injective we may consider its inverse $\xi_{m}^{-1}\colon \xi_{m}(\Z^{d-1})\to \Z^{d-1}$. Remembering that $\Z^{d-1}\cap \xi_{m}(\Z^{d-1})=\emptyset$, we may define $\sigma_{m}\colon \frac{1}{N}\Z^{d-1}\to \frac{1}{N}\Z^{d-1}$ by
	\begin{equation*}
	\sigma_{m}(x)=\begin{cases}
	\xi_{m}(x) & \text{if }x\in \Z^{d-1},\\
	\xi^{-1}_{m}(x) & \text{if }x\in \xi_{m}(\Z^{d-1}),\\
	x & \text{if }x\in \frac{1}{N}\Z^{d-1}\setminus \br{\Z^{d-1}\cup \xi_{m}\br{\Z^{d-1}}}.
	\end{cases}
	\end{equation*}
	It is clear that $\sigma_{m}\colon \frac{1}{N}\Z^{d-1}\to \frac{1}{N}\Z^{d-1}$ is a permutation extending $\xi_{m}$, $\sigma_{m}\circ \sigma_{m}=\id$ and that $\enorm{\sigma_{m}(x)-x}\leq T$, defined by \eqref{eq:xi-id}, for all $x\in\frac{1}{N}\Z^{d-1}$.  
	Let $\Upsilon \colon \R^{d}\to \R^{d}$ be defined by
	\begin{equation*}
	\Upsilon(x)=\begin{cases}
	\Upsilon_{m}(x) & \text{if }x\in \R^{d-1}\times\sqbr{m-\frac{1}{2},m+\frac{1}{2}},\, m\in [H],\\
	x & \text{if }x\in\R^{d-1}\times\br{\R\setminus \br{\frac{1}{2},H+\frac{1}{2}}},
	\end{cases}
	\end{equation*}
	where for each $m\in[H]$ the  $2^{6d+10}N^{2d+1}T^{2d}$- bilipschitz mapping
	\begin{equation*}
	\Upsilon_{m}\colon \R^{d-1}\times\sqbr{m-\frac{1}{2},m+\frac{1}{2}}\to\R^{d-1}\times \sqbr{m-\frac{1}{2},m+\frac{1}{2}}
	\end{equation*}
	is given by Lemma~\ref{lemma:_specific_applied_permute} for the permutation $\sigma_{m}\colon \frac{1}{N}\Z^{d-1}\to \frac{1}{N}\Z^{d-1}$. 
	By Lemma~\ref{lemma:simple_gluing} applied to $\Upsilon$ and the collection of open sets $\br{\R^{d-1}\times\br{m-\frac{1}{2},m+\frac{1}{2}}}_{m\in [H]}$ together with $\R^{d-1}\times\br{-\infty,\frac{1}{2}}$ and $\R^{d-1}\times \br{H+\frac{1}{2},\infty}$ we verify that $\Upsilon$ is bilipschitz with
	\begin{align}
	\bilip\br{\Upsilon}&\leq 2^{11d}N^{7d}\br{T/N^{2}}^{2d}\nonumber \\
	&\leq 2^{11d}\br{2^{10}d^{3/2}L^{3}H^{\frac{1}{d-1}}}^{7d}\br{2^{9}LH^{3}}^{2d}\nonumber \\
	&\leq 2^{99d}d^{21d/2}L^{23d}H^{13d}. \label{eq:bUps}
	\end{align}
	Further, for all $x\in \Z^{d-1}$ and $m\in [H]$, we have
	\begin{equation}\label{eq:transform_to_id}
	\Upsilon\circ \Psi \circ f(x,m)=\Upsilon(\xi_{m}(x),m)=(\sigma_{m}\circ \xi_{m}(x),m)=(x,m).
	\end{equation}
	The desired extension $F\colon\R^{d}\to\R^{d}$ may now be defined as
	\begin{equation*}
	F:=\Psi^{-1}\circ \Upsilon^{-1}.
	\end{equation*}
	Applying this mapping to both sides of \eqref{eq:transform_to_id} we verify that $F(x,m)=f(x,m)$ for all $x\in\Z^{d-1}$ and $m\in [H]$. Moreover, it is clear that $F$ coincides with the identity on $\R^{d-1}\times \set{\frac{1}{2},H+\frac{1}{2}}$. 
	Thus, $F$ is an extension of $f$. Finally, we put together \eqref{eq:bPsi} and \eqref{eq:bUps} to get
	\begin{equation*}
	\bilip(F)\leq\bilip(\Psi)\cdot\bilip(\Upsilon)
	\leq d^{125d}L^{26d}H^{15d}.
	\end{equation*}
\end{proof}

\section{Two conjectures.}\label{sec:conjectures}
In this last section, we prove Theorem~\ref{thm:conjectures_imply_full_solution}, which establishes a relationship between the open problem of bilipschitz extension from separated nets in general dimension $d$ and two conjectures: Conjectures~\ref{conj1} and \ref{conj2}. We emphasise that all the statements, Lemmas~\ref{lemma:gluing}--\ref{lemma:well_separating}, of the present section are independent of any conjectures and are potentially useful in their own right for the future study of the bilipschitz extension problem for $\Z^{d}$. Conjectures~\ref{conj1} and \ref{conj2} only play a role in the final proof, of Theorem~\ref{thm:conjectures_imply_full_solution}.

\begin{lemma}\label{lemma:gluing}
	Let $I$ be a set, $(A_{i})_{i\in I}$ be a pairwise disjoint collection of open subsets $A_{i}$ of $\R^{d}$, $L\geq 1$ and $G\colon \bigcup_{i\in I}\cl{A_{i}}\to \R^{d}$ be a mapping such that the collection $(G(A_{i}))_{i\in I}$ is pairwise disjoint, the restriction $G|_{\bigcup_{i\in I} \partial A_{i}}$ is $L$-bilipschitz and for each $i\in I$ the restriction $G|_{A_{i}}$ is $L$-bilipschitz. Then $G$ is $L$-bilipschitz.
\end{lemma}
\begin{proof}
	From the assumptions it follows that $G$ is injective, so its inverse $G^{-1}\colon \bigcup_{i\in I}\cl{G(A_{i})}\to \R^{d}$ is well-defined. The assumptions of the lemma are also symmetric with respect to interchanging $(G,A_{i})$ and $(G^{-1},G(A_{i}))$, due to Brouwer's Invariance of Domain~\cite[Thm. 2B.3]{Hatcher02}. Therefore, it suffices to verify that $G$ is $L$-Lipschitz.
	
	It is enouch to check the $L$-Lipschitz condition for a pair $x\in A_{i}$ and $y\in A_{j}$ with $i,j\in I$ and $i\neq j$. Note that $[x,y]\cap \partial A_{i}\neq \emptyset$ and choose $z_{1}\in [x,y]\cap \partial A_{i}$. Similarly, choose $z_{2}\in [z_{1},y]\cap \partial A_{j}$. Then, setting $z_{0}=x$ and $z_{3}=y,$ we have
	\begin{equation*}
	\enorm{G(y)-G(x)}\leq \sum_{i=1}^{3}\enorm{G(z_{i})-G(z_{i-1})}\leq L\sum_{i=1}^{3}\enorm{z_{i}-z_{i-1}}= L\enorm{y-x}.
	\end{equation*} 
\end{proof}

\begin{lemma}\label{lemma:far_apart_bilipschitz}
	Let $d\in\N$, $d\geq 2$, $1\leq L<K\leq J$, $T\in\N$, $T\geq \frac{K\sqrt{d}\br{1+JL}}{K-L}$, $f\colon\Z^{d} \to \R^{d}$ be an $L$-bilipschitz mapping and for each $i\in\Z$ let $F_{i}\colon \Z^{d}\cup\br{\R^{d-1}\times \set{Ti+\frac{1}{2}}}\to \R^{d}$ be a $J$-bilipschitz extension of $f$ such that $\bilip\br{F_{i}|_{\R^{d-1}\times\set{Ti+\frac{1}{2}}}}\leq K$. Then the mapping $F\colon \Z^{d}\cup\br{\R^{d-1}\times \br{T\Z+\frac{1}{2}}}\to\R^{d}$ defined by
	\begin{equation*}
	F(x)=F_{i}(x)\quad\text{whenever }x\in\Z^{d}\cup\br{\R^{d-1}\times\set{Ti+\frac{1}{2}}},\, i\in\Z,
	\end{equation*}
	is $J$-bilipschitz and satisfies $\bilip\br{F|_{\R^{d-1}\times\br{T\Z+\frac{1}{2}}}}\leq K$.
\end{lemma}
\begin{proof}
	We show that $F|_{\R^{d-1}\times\br{T\Z+\frac{1}{2}}}$ is $K$-bilipschitz. It is enough to verify the $K$-bilipschitz condition for a pair of points $x\in \R^{d-1}\times\set{Ti+\frac{1}{2}}$, $y\in\R^{d-1}\times\set{Tj+\frac{1}{2}}$, where $i,j\in\Z$ and $i\neq j$. We choose $x', y'\in\Z^{d}$ minimising distance to $x,y$, respectively. Since $\enorm{x-y}\geq T\abs{i-j}\geq T$, it holds that
	\begin{align*}
	\enorm{F(x)-F(y)}&\leq\enorm{f(x')-f(y')}+\enorm{F_i(x)-F_i(x')}+\enorm{F_j(y)-F_j(y')}\\
	&\leq L(\enorm{x-y}+\sqrt{d})+J\sqrt{d}\leq\br{L+\frac{(L+J)\sqrt{d}}{T}}\enorm{x-y}\leq K\enorm{x-y}.
	\end{align*}
	Similarly, we get that
	\begin{align*}
	\enorm{F(x)-F(y)}&\geq\enorm{f(x')-f(y')}-\enorm{F_i(x)-F_i(x')}-\enorm{F_j(y)-F_j(y')}\\
	&\geq\frac{\enorm{x-y}-\sqrt{d}}{L}-J\sqrt{d}\geq\br{\frac{1}{L}-\frac{\sqrt{d}}{TL}-\frac{J\sqrt{d}}{T}}\enorm{x-y}\geq \frac{\enorm{x-y}}{K}.
	\end{align*}
	Hence $F|_{\R^{d-1}\times\br{T\Z+\frac{1}{2}}}$ is $K$-bilipschitz. This, together with the definition of $F$, implies that $F$ is $J$-bilipschitz.
\end{proof}

\begin{lemma}\label{lemma:well_separating}
	Let $d\in\N$, $d\geq 2$, $F\colon \Z^{d}\cup\br{\R^{d-1}\times \br{T\Z+\frac{1}{2}}}\to\R^{d}$ be bilipschitz such that for each $i\in\Z$ the set $\R^{d}\setminus F\br{\R^{d-1}\times\set{Ti+\frac{1}{2}}}$ has two connected components and for every $x,y\in\Z^{d}$
	\begin{align}
	\text{$x$ and $y$ belong to the same} &\text{ connected component of $\R^{d}\setminus \br{\R^{d-1}\times\set{Ti+\frac{1}{2}}}$}\nonumber\\
	&\Updownarrow\label{eq:separation}\\
	\text{$F(x)$ and $F(y)$ belong to the same} &\text{ connected component of $\R^{d}\setminus F\br{\R^{d-1}\times\set{Ti+\frac{1}{2}}}$.}\nonumber
	\end{align}
	Then, for each $i\in\Z$, there is a unique connected open subset $V_{i}$ of $\R^{d}$ with boundary equal to $F\br{\R^{d-1}\times\set{T(i-1)+\frac{1}{2},Ti+\frac{1}{2}}}$. Moreover, the collection $(V_{i})_{i\in\Z}$ has the following properties:
	\begin{enumerate}[(i)]
		\item\label{Vi} $V_{i}\cap V_{j}=\emptyset$ for all $i,j\in\Z$ with $i\neq j$.
		\item\label{images_in_Vi} $\displaystyle F\br{\Z^{d}\cap \br{\R^{d-1}\times\sqbr{T(i-1)+\frac{1}{2},Ti+\frac{1}{2}}}}\subseteq V_{i}$ for each $i\in\Z$.
	\end{enumerate}
\end{lemma}
\begin{proof}
	For each $i\in\Z$ we let $\mc{L}_{i}:=\R^{d-1}\times\set{Ti+\frac{1}{2}}$. For each $i\in\Z$ the set $F(\mc{L}_{i})$ lies in one of the two connected components of $\R^{d}\setminus F(\mc{L}_{i-1})$ and vice-versa. It follows that $\R^{d}\setminus F\br{\mc{L}_{i-1}\cup\mc{L}_{i}}$ has exactly three connected components, exactly one of which has boundary equal to $F(\mc{L}_{i-1}\cup\mc{L}_{i})$. Hence, the sets $V_{i}$, for $i\in\Z$, exist and are well-defined as claimed.
	
	We first verify  \eqref{images_in_Vi}. Fix $i\in\Z$, set 
	\begin{equation*}
	C_{i}:=\Z^{d}\cap \conv\br{\mc{L}_{i-1}\cup\mc{L}_{i}},
	\end{equation*}
	where $\conv$ stands for the convex hull, and let $x\in C_{i}$ and $y\in \Z^{d}\setminus C_{i}$.
	Then, by \eqref{eq:separation}, there is $\sigma\in\set{-1,0}$ such that $F(x)$ and $F(y)$ belong to different connected components of $\R^{d}\setminus F(\mc{L}_{i+\sigma})$. Since $x\in C_{i}$ and $y\in \Z^{d}\setminus C_{i}$ were arbitrary, this proves that the connected components of $\R^{d}\setminus F(\mc{L}_{i-1}\cup \mc{L}_{i})$ determine a partition of $F(\Z^{d})$ which separates $F\br{C_{i}}$ from $F\br{\Z^{d}\setminus C_{i}}$. By \eqref{eq:separation}, the latter set has to intersect both connected components of $\R^{d}\setminus F(\mc{L}_{i+\sigma})$ for both $\sigma=-1$ and $\sigma=0$. This leaves only the connected component $V_{i}$ of $\R^{d}\setminus F(\mc{L}_{i-1}\cup \mc{L}_{i})$ with boundary equal to $F(\mc{L}_{i-1})\cup F(\mc{L}_{i})$ free to accommodate $F\br{C_{i}}$. This proves \eqref{images_in_Vi}.	
	
	To establish \eqref{Vi}, fix $i\in \Z$ and let $U_{+},U_{-}$ be the connected components of $\R^{d}\setminus F\br{\mc{L}_{i}}$ containing $F\br{\br{Ti+1}e_{2}}$ and $F\br{Tie_{2}}$ respectively. If $j>i$ then, according to \eqref{eq:separation}, the points $F\br{\br{Tj+1}e_{2}}$ and $F\br{\br{Ti+1}e_{2}}$ belong to the same connected component $U_{+}$ of $\R^{d}\setminus F\br{\mc{L}_{i}}$, but to different connected components of $\R^{d}\setminus F\br{\mc{L}_{j}}$. We conclude that $F\br{\mc{L}_{j}}\subseteq U_{+}$. Similarly, we can show that $F\br{\mc{L}_{j}}\subseteq U_{-}$ whenever $j<i$. To summarise, we have shown for an arbitrary $i\in\Z$ that all the sets $F(\mc{L}_{j})$ with $j>i$ lie inside one of the connected components of $\R^{d}\setminus F(\mc{L}_{i})$, whilst all the sets $F(\mc{L}_{j})$ with $j<i$ lie in the other. This clearly implies \eqref{Vi}.
\end{proof}
\begin{proof}[Proof of Theorem~\ref{thm:conjectures_imply_full_solution}]
	Suppose \eqref{f_has_extension_F} holds and let $L\geq 1$ and $f\colon \Z^{d}\to \R^{d}$ be an $L$-bilipschitz mapping. Then $P:=P(d,L)$ and a $P$-bilipschitz extension $F\colon \R^{d}\to\R^{d}$ of $f$ are given by \eqref{f_has_extension_F}. Now $K(d,L)=J(d,L)=P(d,L)$ and $G=F$ verify Conjecture~\ref{conj1} for $f$. Hence \eqref{conj1_holds} holds.

	Now suppose that \eqref{conj1_holds} holds and let $L\geq 1$ and $f\colon \Z^{d}\to \R^{d}$ be an $L$-bilipschitz mapping. Let $K:=K(d,L)>L$ and $J:=J(d,L)>L$ be given by Conjecture~\ref{conj1} and let $T:=\ceil{\frac{K\sqrt{d}(1+JL)}{K-L}}$. For each $i\in\Z$ let the $K$-bilipschitz mappings $\wt{G}_{i}\colon \R^{d}\to\R^{d}$ and the $J$-bilipschitz mapping $\wt{F}_{i}\colon\Z^{d}\cup\br{\R^{d-1}\times\set{Ti+\frac{1}{2}}}\to\R^{d}$ be given by Conjecture~\ref{conj1} (appropriately shifted so that $\R^{d-1}\times\set{\frac{1}{2}}$ becomes $\R^{d-1}\times\set{Ti+\frac{1}{2}}$). 
	
	Let $\wt{F}\colon \Z^{d}\cup\br{\R^{d-1}\times\br{T\Z+\frac{1}{2}}}\to\R^{d}$ be the bilipschitz mapping, with
	\begin{equation*}
	\bilip(\wt{F})\leq J,\qquad \bilip(\wt{F}|_{\R^{d-1}\times\br{T\Z+\frac{1}{2}}})\leq K,
	\end{equation*}
	provided by Lemma~\ref{lemma:far_apart_bilipschitz} with $\wt{F}_{i}$ in the role of $F_{i}$ for each $i\in\Z$. Then Lemma~\ref{lemma:well_separating} applies to $\wt{F}$; we get a pairwise disjoint collection of sets $(V_{i})_{i\in\Z}$, where for each $i\in\Z$, $V_{i}$ is the open, connected subset of $\R^{d}\setminus \br{\wt{F}\br{\R^{d-1}\times\set{T(i-1)+\frac{1}{2},Ti+\frac{1}{2}}}}$ with boundary equal to $\wt{F}\br{\R^{d-1}\times\set{T(i-1)+\frac{1}{2},Ti+\frac{1}{2}}}$. Moreover, we have\begin{equation*}
	f\br{\Z^{d}\cap\br{\R^{d-1}\times\sqbr{T(i-1)+\frac{1}{2},Ti+\frac{1}{2}}}}\subseteq V_{i}\qquad\text{for each }i\in\Z.
	\end{equation*}
	
	Observe that Conjecture~\ref{conj2} holds, with the same constant $C(d,L,M)$, independent of $T>0$, if the domain of the mapping $\psi$ is instead taken as $\R^{d-1}\times\set{0,T}$. In this case, the `$\R^{d-1}\times\set{0,1}$' version of the Corollary applies to the $M$-bilipschitz mapping $\R^{d-1}\times\set{0,1}\to\R^{d}$, $x\mapsto \frac{1}{T}\psi(Tx)$ to provide its $C(d,L,M)$-bilipschitz extension $\Psi\colon \R^{d}\to\R^{d}$. Then the mapping $\R^{d}\to\R^{d}$, $x\mapsto T\Psi(x/T)$ will be a $C(d,L,M)$-bilipschitz extension of $\psi\colon \R^{d-1}\times \set{0,T}\to\R^{d}$.  
	
	Note, for each $i\in\Z$, that the mapping $\psi_{i}:=\wt{G}_{i}^{-1}\circ \wt{F}|_{\R^{d-1}\times\set{T(i-1)+\frac{1}{2},Ti+\frac{1}{2}}}$ is $K^{2}$-bilipschitz, coincides with the identity on $\R^{d-1}\times\set{Ti+\frac{1}{2}}$ and with the $K^{2}$-bilipschitz mapping $\wt{G}_{i}^{-1}\circ \wt{G}_{i-1}$ on $\R^{d-1}\times\set{T(i-1)+\frac{1}{2}}$. Thus, for each $i\in\Z$, we may apply the `$\R^{d-1}\times\set{0,T}$' version of Conjecture~\ref{conj2} to extend $\psi_{i}$ to a $C$-bilipschitz mapping $\Psi_{i}\colon \R^{d}\to\R^{d}$, where $C:=C(d,K^{2},K^{2})$.
	
	Set $G_{i}:=\wt{G}_{i}\circ \Psi_{i}$ for each $i\in\Z$. Now we may define $G\colon \R^{d}\to\R^{d}$ as the gluing together of the mappings $G_{i}|_{\set{\R^{d-1}\times\sqbr{T(i-1)+\frac{1}{2},Ti+\frac{1}{2}}}}$, that is
	\begin{equation*}
	G(x):=G_{i}(x)\qquad \text{whenever }i\in\Z\text{ and }x\in \R^{d-1}\times\sqbr{T(i-1)+\frac{1}{2},Ti+\frac{1}{2}}.
	\end{equation*}
	Each restriction $G|_{\R^{d-1}\times\sqbr{T(i-1)+\frac{1}{2},Ti+\frac{1}{2}}}$, with $i\in\Z$, is $KC$-bilipschitz. Moreover, we have that $G|_{\R^{d-1}\times\br{T\Z+\frac{1}{2}}}=\wt{F}|_{\R^{d-1}\times \br{T\Z+\frac{1}{2}}}$. Hence, $G|_{\R^{d-1}\times\br{T\Z+\frac{1}{2}}}$ is $K$-bilipschitz and 
	\begin{equation*}
	G\br{\R^{d-1}\times\br{T(i-1)+\frac{1}{2},Ti+\frac{1}{2}}}=V_{i}\qquad \text{for each }i\in\Z.
	\end{equation*}
	Thus, the conditions of Lemma~\ref{lemma:gluing} are satisfied by the mapping $G$ and the collection of open sets $\br{\R^{d-1}\times\br{T(i-1)+\frac{1}{2},Ti+\frac{1}{2}}}$. We conclude that $G$ is $KC$-bilipschitz.
	
	The proof is now completed by Theorem~\ref{thm:bil_ext_equiv}, since $T\leq \frac{K\sqrt{d}(1+JL)}{K-L}+1$, the $KC$-bilipschitz mapping $G$ and the $J$-bilipschitz mapping $\wt{F}$ verify its condition \eqref{bil_wsep}.	
\end{proof}

\bibliographystyle{plain}
\bibliography{citations}

\begin{thebibliography}{10}

\bibitem{adiceam2016open}
F.~Adiceam, D.~Damanik, F.~G{\"a}hler, U.~Grimm, A.~Haynes, A.~Julien,
  A.~Navas, L.~Sadun, and B.~Weiss.
\newblock {Open problems and conjectures related to the theory of mathematical
  quasicrystals}.
\newblock {\em Arnold Mathematical Journal}, 2(4):579--592, 2016.
\newblock \url{https://doi.org/10.1007/s40598-016-0046-6}.

\bibitem{bilip_ext_sep_nets}
P.~Alestalo, D.~A. Trotsenko, and J.~Väisälä.
\newblock Linear bilipschitz extension property.
\newblock {\em Siberian Mathematical Journal}, 44(6):959--968, 2003.
\newblock \url{https://doi.org/10.1023/B:SIMJ.0000007471.47551.5d}.

\bibitem{alexander_horned_sphere}
J.~W. Alexander.
\newblock {An Example of a Simply Connected Surface Bounding a Region which is
  not Simply Connected}.
\newblock {\em Proceedings of the National Academy of Sciences of the United
  States of America}, 10(1):8--10, 1924.
\newblock \url{http://www.jstor.org/stable/84202}.

\bibitem{Azagra2021}
D.~Azagra, E.~Le~Gruyer, and C.~Mudarra.
\newblock {Kirszbraun’s Theorem via an Explicit Formula}.
\newblock {\em Canadian Mathematical Bulletin}, 64(1):142–153, 2021.
\newblock \url{https://doi.org/10.4153/S0008439520000314}.

\bibitem{Brooks_1941}
R.~L. Brooks.
\newblock On colouring the nodes of a network.
\newblock {\em Mathematical Proceedings of the Cambridge Philosophical
  Society}, 37(2):194–197, 1941.
\newblock \url{https://doi.org/10.1017/S030500410002168X}.

\bibitem{BK1}
D.~Burago and B.~Kleiner.
\newblock Separated nets in {Euclidean} space and {Jacobians} of {biLipschitz}
  maps.
\newblock {\em Geometric and Functional Analysis}, 8:273--282, 1998.
\newblock \url{https://doi.org/10.1007/s000390050056}.

\bibitem{Ciosmak2024}
K.~J. Ciosmak.
\newblock Continuity of extensions of {Lipschitz} maps and of monotone maps.
\newblock {\em Journal of the London Mathematical Society}, 110(5):e70014,
  2024.
\newblock \url{https://doi.org/10.1112/jlms.70014}.

\bibitem{cortez2016some}
M.~Cortez and A.~Navas.
\newblock {Some examples of repetitive, nonrectifiable Delone sets}.
\newblock {\em Geometry \& Topology}, 20(4):1909--1939, 2016.
\newblock \url{https://doi.org/10.2140/gt.2016.20.1909}.

\bibitem{DP_bilip_square}
S.~Daneri and A.~Pratelli.
\newblock A planar bi-lipschitz extension theorem.
\newblock {\em Advances in Calculus of Variations}, 8(3):221--266, 2015.
\newblock \url{https://doi.org/10.1515/acv-2012-0013}.

\bibitem{DiestelGT}
R.~Diestel.
\newblock {\em Graph Theory}.
\newblock Graduate Texts in Mathematics. Springer Berlin Heidelberg, 2025.
\newblock \url{https://doi.org/10.1007/978-3-662-70107-2}.

\bibitem{DK_2dim}
M.~Dymond and V.~Kaluža.
\newblock Planar bilipschitz extension from separated nets, 2025.
\newblock \url{https://arxiv.org/abs/2410.22294}.

\bibitem{Hatcher02}
A.~Hatcher.
\newblock {\em {Algebraic topology}}.
\newblock Cambridge University Press, Cambridge, 2002.
\newblock \url{https://pi.math.cornell.edu/~hatcher/AT/AT.pdf}.

\bibitem{Kirszbraun}
M.~Kirszbraun.
\newblock Über die zusammenziehende und {L}ipschitzsche {T}ransformationen.
\newblock {\em Fundamenta Mathematicae}, 22(1):77--108, 1934.
\newblock \url{http://eudml.org/doc/212681}.

\bibitem{EvaKopecka2012}
E.~Kopecká.
\newblock {Bootstrapping Kirszbraun's extension theorem}.
\newblock {\em Fundamenta Mathematicae}, 217(1):13--19, 2012.
\newblock \url{http://eudml.org/doc/282612}.

\bibitem{Kovalev_bilip_ext_from_line}
L.~V. Kovalev.
\newblock Sharp distortion growth for bilipschitz extension of planar maps.
\newblock {\em Conformal Geometry and Dynamics of the American Mathematical
  Society}, 16(7):124--131, 2012.
\newblock \url{https://doi.org/10.1090/S1088-4173-2012-00243-3}.

\bibitem{Kovalev_optimal_bilip}
L.~V. Kovalev.
\newblock Optimal extension of {Lipschitz} embeddings in the plane.
\newblock {\em Bulletin of the London Mathematical Society}, 51(4):622--632,
  2019.
\newblock \url{https://doi.org/10.1112/blms.12255}.

\bibitem{Elems_of_Lip_topology}
J.~Luukkainen and J.~Väisälä.
\newblock Elements of {Lipschitz} topology.
\newblock {\em Annales Fennici Mathematici}, 2(1):85–122, 1976.
\newblock \url{https://doi.org/10.5186/aasfm.1977.0315}.

\bibitem{Martin_wild_ball}
G.~J. Martin.
\newblock Quasiconformal and bi-lipschitz homeomorphisms, uniform domains and
  the quasihyperbolic metric.
\newblock {\em Transactions of the American Mathematical Society},
  292(1):169--191, 1985.
\newblock \url{https://doi.org/10.2307/2000176}.

\bibitem{Matousek_MetricEmb}
J.~Matou{\v{s}}ek.
\newblock Lecture notes on metric embeddings, 2013.
\newblock \url{https://kam.mff.cuni.cz/~matousek/ba-a4.pdf}.

\bibitem{McM}
C.~T. McMullen.
\newblock Lipschitz maps and nets in {Euclidean} space.
\newblock {\em Geometric and Functional Analysis}, 8:304--314, 1998.
\newblock \url{https://doi.org/10.1007/s000390050058}.

\bibitem{mcshane_extension}
E.~J. McShane.
\newblock {Extension of range of functions}.
\newblock {\em Bulletin of the American Mathematical Society}, 40(12):837 --
  842, 1934.
\newblock \url{https://doi.org/10.1090/S0002-9904-1934-05978-0}.

\bibitem{Ostrovskii_MetricEmb}
M.~I. Ostrovskii.
\newblock {\em Metric Embeddings: Bilipschitz and Coarse Embeddings into Banach
  Spaces}.
\newblock De Gruyter, Berlin, Boston, 2013.
\newblock \url{https://doi.org/10.1515/9783110264012}.

\bibitem{rademacher1919partielle}
H.~Rademacher.
\newblock {\"U}ber partielle und totale differenzierbarkeit von funktionen
  mehrerer variabeln und {\"u}ber die transformation der doppelintegrale.
\newblock {\em Mathematische Annalen}, 79(4):340--359, 1919.
\newblock \url{https://doi.org/10.1007/BF01498415}.

\bibitem{smilansky_solomon_2022}
Y.~Smilansky and Y.~Solomon.
\newblock {A dichotomy for bounded displacement equivalence of Delone sets}.
\newblock {\em Ergodic Theory and Dynamical Systems}, 42(8):2693–2710, 2022.

\bibitem{Sullivan}
D.~Sullivan.
\newblock Hyperbolic geometry and homeomorphisms.
\newblock In James~C. Cantrell, editor, {\em Geometric Topology}, pages
  543--555. Academic Press, 1979.
\newblock \url{https://doi.org/10.1016/B978-0-12-158860-1.50034-4}.

\bibitem{Thomassen-JordanSchoenflies}
C.~Thomassen.
\newblock The {J}ordan-{S}chönflies {T}heorem and the {C}lassification of
  {S}urface.
\newblock {\em The American Mathematical Monthly}, 99(2):116--131, 1992.
\newblock \url{https://doi.org/10.2307/2324180}.

\bibitem{Tukia_planar_Schoenflies}
P.~Tukia.
\newblock The planar {Schönflies} theorem for {Lipschitz} maps.
\newblock {\em Annales Fennici Mathematici}, 5(1):49–72, 1980.
\newblock \url{https://doi.org/10.5186/aasfm.1980.0529}.

\bibitem{tukia1981extension}
P.~Tukia.
\newblock {Extension of quasisymmetric and Lipschitz embeddings of the real
  line into the plane}.
\newblock {\em Annales Fennici Mathematici}, 6(1):89--94, 1981.
\newblock \url{https://doi.org/10.5186/aasfm.1981.0624}.

\bibitem{Tukia_Vaisala-approx_and_extend}
P.~Tukia and J.~Väisälä.
\newblock Lipschitz and quasiconformal approximation and extension.
\newblock {\em Annales Fennici Mathematici}, 6(2):303–342, 1981.
\newblock \url{https://afm.journal.fi/article/view/134419}.

\bibitem{Tukia--Vaisala-Quasiconformal_implies_bilip}
P.~Tukia and J.~Väisälä.
\newblock Bilipschitz extensions of maps having quasiconformal extensions.
\newblock {\em Mathematische Annalen}, 269:561--572, 1984.
\newblock \url{http://eudml.org/doc/163948}.

\bibitem{Tietze--Urysohn}
P.~Urysohn.
\newblock {\"U}ber die {M\"a}chtigkeit der zusammenh{\"a}ngenden {M}engen.
\newblock {\em Mathematische Annalen}, 94(1):262--295, 1925.
\newblock \url{https://doi.org/10.1007/BF01208659}.

\bibitem{Valentine}
F.~A. Valentine.
\newblock {A Lipschitz Condition Preserving Extension for a Vector Function}.
\newblock {\em American Journal of Mathematics}, 67(1):83--93, 1945.
\newblock \url{http://www.jstor.org/stable/2371917}.

\bibitem{whitney}
H.~Whitney.
\newblock {Analytic Extensions of Differentiable Functions Defined in Closed
  Sets}.
\newblock {\em Transactions of the American Mathematical Society},
  36(1):63--89, 1934.
\newblock \url{https://doi.org/10.1090/S0002-9947-1934-1501735-3}.

\end{thebibliography}
\end{document}